\documentclass[11pt]{amsart}
\usepackage{latexsym,amsmath,amssymb,amsfonts,amscd,graphics,appendix,amsxtra}
\usepackage[mathscr]{eucal}
\usepackage[normalem]{ulem}
\usepackage{fullpage}

\usepackage{xr-hyper}

\usepackage{hyperref}
\usepackage[usenames,dvipsnames]{xcolor}

\newtheorem{theorem}{Theorem}[section]
\newtheorem{proposition}[theorem]{Proposition}
\newtheorem{lemma}[theorem]{Lemma}
\newtheorem{claim}[theorem]{Claim}
\newtheorem{corollary}[theorem]{Corollary}

\newtheorem{D}[theorem]{Definition}
\newenvironment{definition}{\begin{D} \rm }{\end{D}}

\newtheorem{R}[theorem]{Remark}
\newenvironment{remark}{\begin{R}\rm }{\end{R}}
\newtheorem{E}[theorem]{Example}

\newtheorem{A}[theorem]{Assumption}
\newenvironment{assumption}{\begin{A} \rm }{\end{A}}

\newcommand{\Dis}{\displaystyle}
\numberwithin{equation}{section}

\def\Zee{\mathbb{Z}}
\def\Q{\mathbb{Q}}

\def\Cee{\mathbb{C}}
\def\Pee{\mathbb{P}}

\def\Ker{\operatorname{Ker}}

\def\Ext{\operatorname{Ext}}

\def\Gr{\operatorname{Gr}}

\def\im{\operatorname{Im}}

\def\Spec{\operatorname{Spec}}

\def\scrO{\mathcal{O}}
\def\spcheck{^{\vee}}
\def\hX{\widehat{X}}
\def\hY{\widehat{Y}}
\def\uhX{\underline{X}^{\#}}
\def\uhY{\underline{Y}^{\#}}
\def\uP{\underline{WP}}
\def\uE{\underline{E}}
\def\uOb{\underline\Omega^\bullet}
\def\uOp{\underline\Omega^p}

\title{Deformations of   Calabi--Yau varieties with $k$-liminal singularities} 

\begin{document}
\author[R. Friedman]{Robert Friedman}
\address{Columbia University, Department of Mathematics, New York, NY 10027}
\email{rf@math.columbia.edu}
\author[R. Laza]{Radu Laza}
\address{Stony Brook University, Department of Mathematics, Stony Brook, NY 11794}
\email{radu.laza@stonybrook.edu}
\begin{abstract}
The goal of this paper is to describe certain nonlinear topological obstructions for the existence of first-order smoothings of mildly singular Calabi--Yau varieties of dimension at least $4$. For nodal Calabi--Yau threefolds,   a necessary and sufficient linear  topological condition for the existence of a first-order smoothing was first given  in \cite{F}. Subsequently, Rollenske--Thomas \cite{RollenskeThomas} generalized this picture to nodal Calabi--Yau varieties of odd dimension, by finding a necessary nonlinear topological condition for the existence of a first-order smoothing. In a complementary direction,  in \cite{FL}, the linear necessary and sufficient conditions of \cite{F} were extended to Calabi--Yau varieties in every dimension with $1$-liminal singularities (which are exactly the ordinary double points in dimension $3$ but not in higher dimensions).  In this paper, we give a common formulation of all of these previous results by establishing  analogues of the  nonlinear topological conditions of \cite{RollenskeThomas} for  Calabi--Yau varieties with weighted homogeneous $k$-liminal hypersurface singularities, a broad class of singularities that includes ordinary double points in odd dimensions.
\end{abstract}

\thanks{Research of the second author is supported in part by NSF grant DMS-2101640.}    
\bibliographystyle{amsalpha}
\maketitle

 \section*{Introduction}
 
The deformation theory of generalized Fano and Calabi--Yau threefolds with ordinary double points (or nodes), or more generally isolated canonical hypersurface singularities, has been extensively studied \cite{F}, \cite{namtop}, \cite{NS}, \cite{namstrata}, \cite{SteenbrinkDB}, \cite{gross_defcy}. This paper is part of a series \cite{FL, FL22b, FL22c, FL23b} which aims to  revisit and sharpen these results and explore generalizations to higher dimensions. A motivating question throughout  has been the problem of understanding the local structure of compactified  moduli spaces of  Calabi--Yau varieties. Let $Y$ be a generalized Calabi--Yau variety in a suitable sense (Definition~\ref{defnCY}).  Two natural questions arise:  (1) Is the first order deformation space of $Y$ unobstructed  (i.e.\  is the moduli space smooth at the point corresponding to $Y$)?  (2) If the  singularities of $Y$ are of some prescribed type,  is there a smoothing of $Y$, i.e.\  a proper flat morphism $\mathcal{Y} \to \Delta$ whose fiber over $0$ is isomorphic to $Y$ and whose general fiber is smooth?  

To put Question (1) in context,   the deformations of a Calabi--Yau manifold $Y$ are unobstructed by the Bogomolov-Tian-Todorov theorem. This result was generalized to the case where $Y$ is allowed to have ordinary double points by Kawamata, Ran, and Tian \cite{Kawamata},   \cite{Ran},  \cite{Tian}. In \cite{FL,FL22c},  this was further extended  to a much wider class of singularities,   \textsl{$1$-Du Bois singularities} (possibly non-isolated  starting in dimension $4$). Turning to Question (2), a natural class of singularities to consider are  isolated Gorenstein canonical (or equivalently rational) singularities. If the singularities are also local complete intersections, there are no local obstructions to smoothability. For isolated hypersurface singularities, there is a natural local condition on first order deformations, i.e.\ deformations over $\Spec \Cee[\varepsilon]$, which we call a \textsl{strong first order smoothing} (Definition~\ref{definestrloc}). This condition guarantees that any deformation of $Y$ over $\Delta$ whose associated first order deformation over $\Spec \Cee[\varepsilon]$ is a strong first order smoothing is a smoothing of $Y$.  If there is a first order deformation of $Y$ which is a strong first order smoothing at every singular point, and if in addition the deformation space of $Y$ is unobstructed, then Question (2) has a positive answer. 

Already in dimension $3$, there is a somewhat paradoxical aspect to Question (2): the ``more rational" the singularities of $Y$, the harder it is to decide if $Y$ is smoothable to first order.  In dimension $3$, this corresponds to the fact that there is a certain (linear) topological constraint in order for the ordinary double points of $Y$ to be smoothable \cite{F}, while no such constraint exists for more complicated rational hypersurface singularities  \cite{NS}, \cite{FL}. In higher dimensions, this phenomenon is even more striking: if $Y$ has rational hypersurface singularities  which are ``not too rational" (not $1$-Du Bois), then  $Y$ is smoothable at least to first order \cite{FL}, but these methods do not apply if the singularities are ``too rational" ($1$-rational).    

A framework for understanding these results  is the theory of \emph{higher Du Bois} and \emph{higher rational singularities}.     Musta\c{t}\u{a}, Popa, Saito along with their collaborators and the authors   have  introduced the notion of \textsl{$k$-Du Bois}  and  \textsl{$k$-rational} singularities for a complex algebraic variety $X$ (for $0\le k\le \dim X$), extending the usual notions of Du Bois  and rational  singularities respectively  (which correspond to the case $k=0$)  \cite{MOPW}, \cite{JKSY-duBois}, \cite{FL22c}. If $X$ has local complete intersection (lci) singularities, then   $k$-rational $\implies$ $k$-Du Bois $\implies$ $(k-1)$-rational \cite{ChenDirksM}, \cite{FL22c}, \cite{FL22d}.  Thus, as $k$ increases, the singularities become  milder:  A  local complete intersection singularity which is $k$-Du Bois with $k > \frac12(\dim X-1)$ is smooth, and it is an ordinary double point if $k = \frac12(\dim X-1)$.   Varieties with $k$-rational and $k$-Du Bois singularities satisfy various vanishing and non-vanishing results (e.g.\ \cite{S-Van}, \cite{SteenbrinkDB}, \cite{MP-Inv}, \cite{FL22d}), which in turn are closely related to the deformation theory of Calabi--Yau varieties in case $k=1$ \cite{FL}. In particular, the deformation theory of $Y$ is especially well-behaved when the singularities are $1$-Du Bois  but not $1$-rational. In this case, Question (1) has a positive answer and, for Question (2),  there is a necessary and sufficient condition for the existence of a strong first order smoothing in  case the singular points of $Y$  are isolated hypersurface singularities. 

As noted above,  the  methods of \cite{FL} unfortunately say nothing about the answer to Question (2) if the singularities are $1$-rational. 
On the positive side,  for odd-dimensional Calabi--Yau varieties $Y$ (of dimension at least $5$) with only ordinary double points, Rollenske and Thomas found  a  \emph{non-linear} obstruction  to the existence of a first-order smoothing of $Y$ \cite{RollenskeThomas}, which we state more precisely below (Theorem~\ref{RTthm}).  To find an appropriate generalization of this result, we make the following definition  
 \cite[Definition 6.10]{FL22d},  Definition~\ref{def0.3}:  An  isolated hypersurface singularity is \textsl{$k$-liminal}  if it is $k$-Du Bois, but not $k$-rational.   In dimension $3$,  the only $1$-liminal singularities are ordinary double points.  More generally  in odd dimension $2k+1$,  the only $k$-liminal singularities are ordinary double points. However, ordinary double points in even dimensions are not $k$-liminal for any value of $k$.  By Lemma~\ref{ex0.4.1} below, for every $n \ge 3$, there exist $k$-liminal singularities of dimension $n$ $\iff$ $0\le k \le \Dis \left[\frac{n-1}{2}\right]$. In particular, for every $n \ge 3$, there exist $k$-liminal singularities of dimension $n$ for some  $k\ge 1$.  Thus, $k$-liminal singularities are important boundary/transition cases and are a far-reaching  generalization of ordinary double points in  odd dimensions. 

Since  the ordinary double points are exactly the  $k$-liminal lci singularities in dimension $2k+1$, the Rollenske--Thomas  theorem can then be rephrased as follows:  If $Y$ is a Calabi--Yau variety of dimension $n=2k+1$ with only $k$-liminal lci singularities, there is a  topological obstruction   to the existence of a strong first order smoothing of $Y$ (i.e.\ a necessary  condition for the existence of a strong first order smoothing) which is (roughly) $k$-linear.  (In dimension $3$, the obstruction is a linear condition, and it is also  sufficient \cite{F}.) The main result of this paper is a generalization of the Rollenske--Thomas theorem to the case where $Y$ is a Calabi--Yau variety with isolated hypersurface weighted homogeneous $k$-liminal singularities. \\

To  explain our results in more detail, we begin with the following definition:   

 \begin{definition}\label{defnCY} A \textsl{canonical Calabi--Yau variety $Y$} is a   compact analytic variety $Y$  with at worst canonical Gorenstein (or equivalently rational Gorenstein) singularities, such that $\omega_Y\cong \scrO_Y$, and such that either $Y$ is a scheme or $Y$ has only isolated singularities and the $\partial\bar\partial$-lemma holds for some resolution of $Y$.
\end{definition} 
 
For a    compact analytic variety $Y$ with at worst ordinary double point singularities, recall that a \textsl{first-order deformation of $Y$} is a flat proper morphism $f\colon \mathcal{Y} \to \Spec \Cee[\epsilon]$, together with an isomorphism from the fiber over $0$ to $Y$, and these are classified by $\mathbb{T}_Y^1 = \Ext^1(\Omega^1_Y, \scrO_Y)$. Given a class $\theta \in   \mathbb{T}_Y^1$, its image in $H^0(Y; T^1_Y) = \bigoplus_{x\in Y_{\text{\rm{sing}}}}T^1_{Y,x}$ measures the first-order change to the singularities of $Y$, and $\theta$ is a \textsl{first-order smoothing of $Y$} if the image of $\theta$ in $T^1_{Y,x}\cong \Cee$ is nonzero for every $x\in Y_{\text{\rm{sing}}}$.  Then by \cite[\S4]{F} (also \cite[Prop. 8.7]{FriedmanSurvey}), we have: 
 
 \begin{theorem}\label{dim3criterion} Suppose that $Y$ is a canonical Calabi--Yau variety of dimension $3$ whose only singularities are ordinary double points.  Let $\pi\colon Y'\to Y$ be a small resolution of the singularities of $Y$, so that $\pi^{-1}(x) =C_x\cong \Pee^1$ for every $x\in Y_{\text{\rm{sing}}}$, and let $[C_x]$ be the fundamental class of $C_x$ in $H^2(Y'; \Omega^2_{Y'})$. Then a first-order smoothing of $Y$ exists $\iff$ there exist $a_x\in \Cee$, $a_x \neq 0$ for every $x$, such that  $\sum_{x\in Y_{\text{\rm{sing}}}} a_x[C_x] =0$ in $H^2(Y'; \Omega^2_{Y'})$.
\end{theorem}

Next we  describe the partial extension of Theorem~\ref{dim3criterion}  to all odd dimensions $n = 2k+1 \ge 3$ due to Rollenske--Thomas. For $n > 3$, there is no small resolution of an ordinary double point. Instead, consider the standard blowup of a node.   The exceptional divisor is an even dimensional quadric, whose primitive cohomology is generated by  the difference $[A]-[B]$, where $A$ and $B$ are two complementary linear spaces of dimension $k$ such that $A\cdot B =1$. For $Y$ a projective variety of dimension $2k+1$ whose only singular points are nodes and $\pi: \hY \to Y$ a standard resolution as above,   for each $x\in Y_{\text{\rm{sing}}}$, there is thus a class $[A_x]-[B_x]\in H^{k+1}(\hY;\Omega^{k+1}_{\hY})$.  The following is equivalent to the necessity part of  Theorem~\ref{dim3criterion} in dimension $3$ and generalizes it   to all odd dimensional nodal canonical Calabi--Yau varieties \cite{RollenskeThomas}:

\begin{theorem} \label{RTthm} Suppose that $Y$ is a canonical Calabi--Yau variety of odd dimension $n=2k+1$ whose only singularities are ordinary double points and let $\hY \to Y$ be a standard resolution as above. Then there exist identifications   $T^1_{Y,x}\cong \Cee$  such that the following holds: If $\theta$ is a  first-order smoothing of $Y$ with  image in $T^1_{Y,x}$ equal to $\lambda_x\in \Cee$ via the above  isomorphisms $T^1_{Y,x}\cong \Cee$, then, with notation as above,  
 \begin{equation}\label{eqn1}
 \sum_{x\in Y_{\text{\rm{sing}}}} \lambda_x^k ([A_x]-[B_x])=0 
 \end{equation} 
   in $H^{k+1}(\hY;\Omega^{k+1}_{\hY})$.
\end{theorem}

 We can interpret Theorem~\ref{RTthm} in the following way. First, if there exists a first-order smoothing of $Y$, then the classes $[A_x]-[B_x]$ are not linearly independent in $H^{k+1}(\hY;\Omega^{k+1}_{\hY})$, and in fact satisfy a linear relation whose coefficients are all nonzero. Second, the image of $\mathbb{T}_Y^1$ in $H^0(Y; T^1_Y)$, which is a vector subspace of $H^0(Y; T^1_Y)$, is contained in the subvariety of $H^0(Y; T^1_Y)$ defined by the nonlinear Equation~\ref{eqn1}, which is roughly speaking an intersection of affine varieties of Fermat type.  
 
 \medskip
 
 The goal of this paper is to   generalize Theorem~\ref{RTthm}.  To state the result, let $Y$ be as before a compact analytic variety with isolated singularities.  If $x\in Y_{\text{\rm{sing}}}$ is a singular point, let $\pi \colon \hY \to Y$ be some log resolution of $Y$ and let $E_x =\pi^{-1}(x)$ be the exceptional divisor over $x$. In the case of ordinary double points,   $\dim T^1_{Y, x} = 1$ for a singular point and   there are distinguished classes $[A_x]-[B_x] \in H^{k+1}(\hY;\Omega^{k+1}_{\hY})$ which are defined locally around the singular points. In general, $\dim T^1_{Y,x}\neq 1$, so we must define  the types of smoothings to which our methods will apply: 

\begin{definition}\label{definestrloc}  Let $(X,x)$ be the germ of an isolated hypersurface singularity, so that $T^1_{X,x} = \scrO_{X,x}/J$ is a cyclic $\scrO_{X,x}$-module. Thus $\dim T^1_{X,x}/\mathfrak{m}_xT^1_{X,x} = 1$. Then an element  $\theta_x\in T^1_{X,x}$ is a \textsl{strong first-order smoothing} if $\theta_x\notin \mathfrak{m}_xT^1_{X,x}$.  In case $x$ is an ordinary double point, $\theta_x\in T^1_{X,x}$ is a  strong first-order smoothing $\iff$ $\theta_x\neq 0$. For a compact $Y$ with only isolated hypersurface singularities, a first-order deformation $\theta\in \mathbb{T}^1_Y$ is a \textsl{strong first-order smoothing} if the image $\theta_x$ of $\theta$ in $T^1_{X,x}$ is a strong first-order smoothing for every $x\in Y_{\text{\rm{sing}}}$. A standard argument (e.g.\ \cite[Lemma 1.9]{FL}) shows that,  if  $f\colon \mathcal{Y}\to \Delta$ is a deformation of $Y$ over the disk, then its Kodaira-Spencer class $\theta$ is a strong first-order smoothing $\iff$ $\mathcal{Y}$ is smooth and in particular  the nearby fibers $Y_t =f^{-1}(t)$, $0< |t| \ll 1$, are smooth.
\end{definition} 

\begin{remark}\label{1DBsmooths} For $k \ge 1$, a  $k$-liminal singularity is in particular $1$-Du Bois. Hence, by \cite[Corollary 1.5]{FL22c}, a canonical Calabi--Yau variety $Y$ with  only isolated $k$-liminal hypersurface singularities has unobstructed deformations. 
 In particular if there exists a strong first-order smoothing of $Y$, then $Y$ is smoothable. 
\end{remark}

 To deal with the correct generalization of the class $[A_x]-[B_x]$, recall that, for each $x\in Y_{\text{\rm{sing}}}$ (assumed throughout to be an isolated hypersurface singularity), 
 we have the corresponding link $L_x$ at $x$.  There is a natural mixed Hodge structure on $H^\bullet(L)$ (see e.g.\ \cite[\S6.2]{PS}). Moreover, for all $k$, there is a natural map 
 \begin{equation}\varphi\colon \Gr^{n-k}_FH^n(L_x) \to H^{k+1}(\hY;\Omega^{n-k}_{\hY})\end{equation}  given as the composition
\begin{gather*}
\Gr^{n-k}_FH^n(L_x) =H^k(E_x;\Omega^{n-k}_{\hY}(\log E_x)|E_x) \to \Gr^{n-k}_FH^{n+1}_E(\hY) = H^k(E_x;\Omega^{n-k}_{\hY}(\log E_x)/\Omega^{n-k}_{\hY})\\
\xrightarrow{\partial}  H^{k+1}(\hY;\Omega^{n-k}_{\hY}).
\end{gather*}
In case there is a Hodge decomposition for $\hY$ (for example if $\hY$ is  K\"ahler or more generally satisfies the $\partial\bar\partial$-lemma), the above maps are consistent in the obvious sense with the topological maps
$$H^n(L_x) \to H^{n+1}_{E_x}(\hY) \to H^{n+1}(\hY),$$
where via Poincar\'e duality the map $H^n(L_x) \to  H^{n+1}(\hY)$ is the same as the natural map  $H_{n-1}(L_x) \to  H_{n-1}(\hY)$. In the special case where $x$ is an ordinary double point and $n=2k+1$,  $\dim H^n(L_x) = 1$,  so that $H^n(L_x) =\Cee \varepsilon_x$ for some $\varepsilon_x \in H^n(L_x)$, 
and, for an appropriate choice of  $\varepsilon_x$,
 $\varphi(\varepsilon_x) = [A_x]-[B_x]\in H^{k+1}(\hY;\Omega^{n-k}_{\hY}) =  H^{k+1}(\hY;\Omega^{k+1}_{\hY})$.  
\medskip

 The link of a $k$-liminal singularity is formally  analogous to that  of an ordinary double point  in odd dimensions, by  the following result,  essentially due to  Dimca-Saito \cite[\S4.11]{DS} (cf.\ also \cite[Corollary 6.14]{FL22d}):

\begin{theorem}\label{introlink}  If $(X,x)$ is the germ of an isolated $k$-liminal hypersurface singularity and $L$ is the corresponding link, then $\dim \Gr^{n-k}_FH^n(L) =1$.
\end{theorem}

 For $1$-liminal  singularities, we showed   \cite[Lemma 5.6, Corollary 5.12]{FL}  that there is a necessary and sufficient linear condition   for there to exist a strong first-order smoothing of $Y$, and hence an actual smoothing by Remark~\ref{1DBsmooths}. 
  This statement (see Theorem~\ref{1limthm} below for a precise version) can be viewed as a  natural generalization of Theorem~\ref{dim3criterion}.  The main results of this paper, Theorem~\ref{mainthma} and Corollary~\ref{last},  are then   further generalizations which apply to all  weighted homogeneous $k$-liminal singularities. However, as  in Theorem~\ref{RTthm}, we are only able to obtain \emph{necessary} conditions for $k \ge 2$: 

\begin{theorem}\label{intromain}  Let $Y$ be a canonical Calabi--Yau variety of dimension $n$ with isolated $k$-liminal  weighted homogeneous hypersurface singularities and $k\ge 1$. For each singular point $x\in Y$, let $L_x$ be the link at $x$ and write $\Gr^{n-k}_FH^n(L_x) = H^k(E_x; \Omega^{n-k}_{\hY}(\log E_x)|E_x) =\Cee\cdot \varepsilon_x$ for some choice of a generator $\varepsilon_x$.  Let $\varphi\colon \Gr^{n-k}_FH^n(L) =\bigoplus_{x\in Y_{\text{\rm{sing}}}}\Gr^{n-k}_FH^n(L)\to H^{k+1}(\hY; \Omega^{n-k}_{\hY})$ be the natural map. 

Finally, for  each $x\in Y_{\text{\rm{sing}}}$, fix an identification $T_{Y,x}^1/\mathfrak{m}_x T_{Y,x}^1\cong \Cee$.  Then, for each $x\in Y_{\text{\rm{sing}}}$ there exist $c_x\in \Cee^*$ with the following property: If $\theta \in \mathbb{T}^1_Y$ induces $\lambda_x\in \Cee$, then 
$$\sum_{x\in Y_{\text{\rm{sing}}}} c_x\lambda_x^k\varphi(\varepsilon_x) =0 \in H^{k+1}(\hY; \Omega^{n-k}_{\hY}).$$
In particular, if a strong first-order smoothing of $Y$ exists, then the classes $\varphi(\varepsilon_x)$ are not linearly independent.
\end{theorem}

 In some sense, the proof of Theorem~\ref{intromain} follows the main outlines of \cite{RollenskeThomas}. A key aspect of our arguments is that by restricting to weighted homogeneous singularities, we can work as if there exists a log resolution with a single (smooth) exceptional divisor $E$ as in loc.\ cit.\ More precisely, for $Y$ with such singularities, there is the weighted blowup, i.e.\ an orbifold resolution of singularities $Y^{\#} \to Y$ whose exceptional divisors $E$ are smooth divisors in the sense of orbifolds. There are  stacks  naturally associated to   $Y^{\#}$ and $E$, a picture which is worked out in detail in \cite[\S3]{FL} (whose methods we use systematically). Thus we can proceed as if $Y^{\#}$ and $E$ were smooth and use the familiar numerology of hypersurfaces in weighted projective space.   It would be interesting to generalize the proof of Theorem~\ref{intromain} to the case where the singularities are not necessarily weighted homogeneous.

\medskip

The outline of this paper is as follows. In \S\ref{ss1.1}, we collect some necessary preliminaries about isolated singularities. $k$-liminal singularities are defined in \S\ref{ss1.2}, and the stack point of view is recalled in \S\ref{ss1.3}. Section \ref{s2} deals with the geometry of  $k$-liminal weighted homogeneous singularities and establishes the existence of a nonzero homogeneous pairing between two one-dimensional vector spaces. In \S\ref{ss3.1}, this construction is globalized to establish Theorem~\ref{intromain} (Theorem~\ref{mainthma} and Corollary~\ref{last}).  There is also a brief discussion in \S\ref{ss3.2} of the interplay between the Hodge theory of $Y$ or of $\hY$ and of a smoothing $Y_t$ of $Y$.

\subsection*{Acknowledgements} It is a pleasure to thank Johan de Jong and Richard Thomas for their comments and suggestions. We would also like to thank the referee for a careful reading of the paper and several helpful suggestions. 

\section{Preliminaries}\label{s1}

\subsection{Some general Hodge theory}\label{ss1.1}  Let $X$ be a contractible Stein neighborhood of the isolated  singularity $x$ of dimension $n\ge 3$, and let $\pi\colon \hX \to X$ be a good (log) resolution, i.e.\ $\pi$ is a resolution of singularities,  and  $E =\pi^{-1}(x)$  (with its reduced structure) is a  divisor with  simple normal crossings. For every coherent sheaf $\mathcal{F}$ on $\hX$, $H^i(\hX; \mathcal{F}) \cong H^0(X; R^i\pi_*\mathcal{F})$. Let $U=X-\{x\} = \hX -E$. In the global setting,  $Y$ will denote  a projective variety  of dimension $n$ with isolated singularities,  $Z = Y_{\text{\rm{sing}}}$ the singular locus of $Y$, and $\pi\colon \hY \to Y$   a good (log) resolution at each singular point.   We will also use $E$ to denote   the exceptional divisor in this context, i.e.\ $E = \pi^{-1}(Z)$, again viewed as a reduced divisor, and $V = \hY - E = Y- Z$.  Instead of assuming that $Y$ is projective,  it is more generally enough to assume that $Y$  has a resolution satisfying the $\partial\bar\partial$-lemma. 

\begin{lemma}\label{biratinv} With $Y$ and  $\pi\colon \hY \to Y$ as above, and for all $p,q$,  the groups $H^q(\hX; \Omega^p_{\hX}(\log E))$, $H^q(\hX; \Omega^p_{\hX}(\log E)(-E))$, $H^q(\hY; \Omega^p_{\hY}(\log E))$, and $H^q(\hY; \Omega^p_{\hY}(\log E)(-E))$ are all independent of the choice of resolution.
\end{lemma}
\begin{proof} The independence of $H^q(\hY; \Omega^p_{\hY}(\log E))$ is a result of Deligne \cite[3.2.5(ii)]{DeligneHodgeII}. The independence of $H^q(\hY; \Omega^p_{\hY}(\log E)(-E))$ then follows because $H^q(\hY; \Omega^p_{\hY}(\log E)(-E))$ is Serre dual to $H^{n-q}(\hY; \Omega^{n-p}_{\hY}(\log E))$. The local results for $\hX$ can then be reduced to this case (cf.\ \cite[Remark 3.15]{FL}). 
\end{proof}

\begin{remark}\label{duality} In case $Y$ is projective, we can understand the birational invariance as follows: Let $\uOb_{Y,Z}$ be the relative filtered de Rham complex as defined by Du Bois \cite{duBois}.  By \cite[Th\'eor\`eme 2.4]{duBois}, $\uOb_{Y,Z}$ is an invariant of $Y$ as an object in the filtered derived category,  and the corresponding Hodge spectral sequence degenerates at $E_1$ in case $Y$ is projective. By \cite[Example 7.25]{PS}, $\uOp_{Y,Z} \cong R\pi_*\Omega^p_{\hY}(\log E)(-E)$. Applying the Leray spectral sequence for hypercohomology gives
$$\Gr^p_FH^{p+q}(Y,Z) = \mathbb{H}^q(Y; \uOp_{Y,Z}) = H^q(\hY; \Omega^p_{\hY}(\log E)(-E)).$$
Hence  $H^q(\hY; \Omega^p_{\hY}(\log E)(-E)) = \Gr^p_FH^{p+q}(Y,Z)$ does not depend on the choice of a resolution.

Note that from the exact sequence
$$\cdots  \to H^{i-1}(Z) \to    H^i(Y,Z) \to H^i(Y) \to H^i(Z) \to \cdots,$$
$ H^i(Y,Z) \cong  H^i(Y)$ except for $i=0,1$ since $\dim Z =0$. Moreover, the hypercohomology of the exact sequence
$$0 \to \Omega^\bullet_{\hY}(\log E)(-E) \to \Omega^\bullet_{\hY} \to \Omega^\bullet_E/\tau^\bullet_E\to 0$$
gives the Mayer--Vietoris sequence, an  exact sequence of mixed Hodge structures:
$$\cdots \to H^{i-1}(E)\to H^i(Y,Z) \to H^i(\hY) \to H^i(E) \to \cdots.$$
Finally, the duality between $\mathbb{H}^\bullet(\hY;  \Omega^\bullet_{\hY}(\log E)(-E))$ and $\mathbb{H}^\bullet(\hY;  \Omega^\bullet_{\hY}(\log E))$ corresponds to Poincar\'e duality (cf.\ \cite[\S5.5, B.21, B.24]{PS}) 
$$H^i(Y,Z) \cong H^i_c(Y-Z) \cong (H^{2n-i}(Y-Z))\spcheck(-n) = (H^{2n-i}(\hY-E))\spcheck(-n).$$
\end{remark} 

 \begin{lemma}\label{lemma3.3}  With $Y$ and  $\pi\colon \hY \to Y$ as above,  the map 
 $$\Gr^{n-k}_FH^{n+1}(Y) = H^{k+1}(\hY; \Omega^{n-k}_{\hY}(\log E)(-E)) \to H^{k+1}(\hY; \Omega^{n-k}_{\hY})$$ is injective for all $k\ge 0$.
 \end{lemma}
 \begin{proof}  We have the exact sequence
  $$H^k(\hY; \Omega^{n-k}_{\hY}) \to H^k(E; \Omega^{n-k}_E/\tau^{n-k}_E) \to H^{k+1}(\hY; \Omega^{n-k}_{\hY}(\log E)(-E)) \to H^{k+1}(\hY; \Omega^{n-k}_{\hY}).$$
  By semipurity in the local setting \cite[(1.12)]{Steenbrink}, the map $H^n_E(\hX) \to H^n(E)$ is an isomorphism. Since it factors by excision as $H^n_E(\hX) \cong H^n_E(\hY) \to H^n(\hY) \to H^n(E)$, the map $H^n(\hY) \to H^n(E)$ is therefore surjective, and hence, by strictness of morphisms, so is the map 
  $$\Gr^{n-k}_FH^n (\hY) = H^k(\hY; \Omega^{n-k}_{\hY}) \to  \Gr^{n-k}_F H^n(E) = H^k(E; \Omega^{n-k}_E/\tau^{n-k}_E) .$$
  Thus the map $H^{k+1}(\hY; \Omega^{n-k}_{\hY}(\log E)(-E)) \to H^{k+1}(\hY; \Omega^{n-k}_{\hY})$ is injective.
 \end{proof}

\subsection{$k$-Du Bois, $k$-rational, and $k$-liminal singularities}\label{ss1.2}The $k$-Du Bois and $k$-rational singularities, natural extensions of Du Bois and rational singularities respectively (the case $k=0$), were recently introduced by \cite{MOPW}, \cite{JKSY-duBois}, \cite{KL2}, \cite{FL22c}, and \cite{MP-rat}. The relevance of these classes of singularities (especially for $k=1$) to the deformation theory of singular Calabi--Yau and Fano varieties is discussed in \cite{FL}, which additionally singles out the {\it $k$-liminal singularities} (for $k=1$) as particularly relevant to the deformation theory of such varieties. The $k$-liminal singularities should be understood as the frontier case between $(k-1)$-rational and $k$-rational. For the convenience of the reader, we summarize the relevant facts for these classes of singularities.

\begin{definition}\label{def0.3}  Let $(X,x)$ be the germ of an isolated local complete intersection (lci) singularity of dimension $n\ge 3$ and let $\pi\colon \hX \to X$ be a good resolution with exceptional divisor $E$. Then $X$ is \textsl{$k$-Du Bois} if $R^i\pi_*\Omega^p_{\hX}(\log E)(-E) = 0$ for $i> 0$ and $p\le k$, and is \textsl{$k$-rational} if $R^i\pi_*\Omega^p_{\hX}(\log E) = 0$ for $i> 0$ and $p\le k$. By \cite{FL22c}, \cite{MP-rat}, if $(X,x)$ is $k$-rational, then it is $k$-Du Bois and by \cite{FL22d}, \cite{ChenDirksM},  if $(X,x)$ is $k$-Du Bois, then it is $(k-1)$-rational.

Finally, $(X,x)$ is \textsl{$k$-liminal} if it is $k$-Du Bois but not $k$-rational. In this case, if $X$ is a hypersurface singularity, then  $\dim \Gr^{n-k}_FH^n(L) =1$, by Theorem~\ref{introlink}. 
\end{definition}

The following collects some basic facts about $k$-liminal singularities:

\begin{lemma}\label{ex0.4}  Let $X$ be the germ of an isolated hypersurface singularity.
\begin{enumerate} 
\item[\rm(i)] If  $\dim X =3$ and $X$ is  not smooth,  then $X$ is not $1$-rational, and $X$ is $1$-liminal $\iff$ $X$ is $1$-Du Bois  $\iff$ $X$ is an ordinary double point.  
\item[\rm(ii)] More generally, if $X$ is a $k$-Du Bois singularity and  $k > \frac12(n-1)$, then $X$ is smooth. If $\dim X= 2k+1$ and $X$ is not smooth, then $X$ is $k$-Du Bois $\iff$  $X$ is $k$-liminal $\iff$ $X$ is an ordinary double point.

\item[\rm(iii)] Suppose that  $X$ is weighted homogeneous. Viewing $X$ as locally analytically isomorphic to the subvariety $\{f=0\}$ of $(\Cee^{n+1}, 0)$, where   $\Cee^*$ acts on $\Cee^{n+1}$ with weights $a_1, \dots, a_{n+1} \ge 1$, and $f$ is weighted homogeneous of degree $d$, define $w_i = a_i/d$. Then:
 \begin{enumerate}
 \item  $X$ is $k$-Du Bois $\iff$ $\sum_{i=1}^{n+1} w_i \ge k+1$.
 \item  $X$ is $k$-rational $\iff$ $\sum_{i=1}^{n+1} w_i > k+1$.
  \item  $X$ is $k$-liminal $\iff$ $\sum_{i=1}^{n+1} w_i = k+1$.
 \end{enumerate}
  \end{enumerate}
\end{lemma} 
\begin{proof} (i) This is  a result of Namikawa-Steenbrink \cite[Theorem 2.2]{NS} (cf.\ also \cite[Corollary 6.12]{FL22d}).  

\smallskip
\noindent (ii) This is  \cite[Corollary 6.3]{DM} (cf.\ also \cite[Corollary 4.4]{FL22d}). 

\smallskip
\noindent (iii) This is a result of Saito \cite[(2.5.1)]{SaitoV} (see also  \ \cite[Corollary 6.8]{FL22d}). 
   \end{proof} 

\begin{remark}\label{klimremark}  (i) By definition, a $0$-liminal singularity is $0$-Du Bois, i.e.\ Du Bois in the terminology of \cite{Steenbrink}, but not rational. Thus these singularities fall outside the scope of this paper. If $X$ is an isolated normal Gorenstein  surface singularity which is Du Bois but not rational, then by \cite[3.8]{Steenbrink} $X$ is either a simple elliptic or a cusp singularity. Such singularities are known to be deeply connected to degenerations of $K3$ surfaces. In \cite{FL23b}, we explore the analogous picture for Calabi--Yau varieties in higher dimensions in case $Y$ has  hypersurface singularities. 

\smallskip
\noindent (ii)  Assume that $X$ is a weighted homogeneous hypersurface singularity.  If $X$ is the cone over a smooth hypersurface $E$ of degree $d$ in $\Pee^n$, then, by Lemma~\ref{ex0.4}(iii),  the $k$-liminal condition is $n+1 = d(k+1)$, and in particular $n+1$ is divisible by $d$ and by $k+1$. Thus, these examples are somewhat sparse. By Theorem~\ref{introlink}, the Hodge structure on $H^{n-1}(E)$ is (up to a Tate twist) of Calabi--Yau type. Primarily for this reason, such hypersurfaces are exceptions to Donagi's proof for generic Torelli (\cite{Donagi}; cf.\ Voisin \cite{Voisin} for recent work along these lines). 
\end{remark} 

Despite Remark~\ref{klimremark}(ii) above, there are many examples of isolated  weighted homogeneous $k$-liminal singularities: 

\begin{lemma}\label{ex0.4.1}  For all $k$ with $\Dis 1\le k \le \left[\frac{n-1}{2}\right]$, there exists an isolated  weighted homogeneous $k$-liminal singularity   given by a diagonal hypersurface $f(z) = z_1^{e_1} + \cdots + z_{n+1}^{e_{n+1}}$.
\end{lemma} 
\begin{proof} Given $k$ such that $\Dis 1\le k \le \left[\frac{n-1}{2}\right]$,  let  $f(z) = z_1^{e_1} + \cdots + z_{n+1}^{e_{n+1}}$. First suppose that  $n=  2a+1$ is odd, so $\Dis\left[\frac{n-1}{2}\right] = a$. Then  choose $2\ell$ of the $e_i$ equal to $2$ and the remaining $n+1 - 2\ell=2(a+1-\ell)$ equal to  $\Dis \frac{n+1 - 2\ell}{2}= a+1-\ell$. Here $0\le \ell \le a-1$ because the value $\ell = a$ would give some $e_i = 1$. Then 
  $$\sum_i w_i = \sum_i \frac{1}{e_i} = \frac12(2\ell)  + (n+1 - 2\ell)\left(\frac{2}{n+1 - 2\ell}\right) = \ell + 2,$$
  and hence $k=\sum_i w_i -1=\ell +1$ can take on all possible values from $1$ to $a$.
  
  Similarly, if $n=2a$ is even, so that $\Dis\left[\frac{n-1}{2}\right] = a-1$, and $1\le \ell \le a-2$, choose $2\ell-1$ of the $e_i$ to be $2$, $2$ of the $e_i$ to be $4$, and the remaining $n+1 - (2\ell+1) = 2a -2\ell$ to be  $\Dis \frac{n+1 - (2\ell+1)}{2} = a-\ell$. Then
   $$\sum_i w_i = \sum_i \frac{1}{e_i} = \frac12(2\ell-1)  + \frac12 + (n+1 - (2\ell+1))\left(\frac{2}{n+1 - (2\ell+1)}\right) = \ell + 2,$$
  and hence $k=\sum_i w_i -1=\ell +1$ can take on all possible values from $2$ to $a-1$. For the remaining possibility $k=1$, take $n-1=2a-1$ of the $e_i$ equal to $a$ and the remaining two equal to $2a$ to get $\sum_iw_i = 2$ and hence $k=1$.
\end{proof}

The following then generalizes \cite[2.6]{RollenskeThomas}:

\begin{lemma}\label{0.2.1} If the singularities of $X$ are isolated  $1$-Du Bois lci singularities, then  $H^0(X;T^1_X) \cong H^1(\hX; \Omega^{n-1}_{\hX}(\log E))$. In the global case,
$\mathbb{T}^1_Y \cong H^1(\hY; \Omega^{n-1}_{\hY}(\log E))$, compatibly with the map $\mathbb{T}^1_Y \to H^0(Y;T^1_Y)$ and restriction, i.e.\  the following diagram commutes:
$$\begin{CD}
  H^1(\hY; \Omega^{n-1}_{\hY}(\log E)) @>>> H^0(Y;R^1\pi_*\Omega^{n-1}_{\hY}(\log E)) \\
  @V{\cong}VV @VV{\cong}V \\
\mathbb{T}^1_Y @>>>  H^0(Y;T^1_Y) .
\end{CD}$$
\end{lemma}
\begin{proof} First, by a result of Schlessinger (see e.g.\ \cite[Lemma 1.16]{FL}), $H^0(X;T^1_X)\cong H^1(U; T^0_X|U)$. Clearly  $ H^1(U; T^0_X|U)  = H^1(U;\Omega^{n-1}_{\hX}(\log E)|U)$. The local cohomology sequence gives
$$H^1_E(\hX; \Omega^{n-1}_{\hX}(\log E)) \to H^1(\hX; \Omega^{n-1}_{\hX}(\log E)) \to H^0(X;T^1_X) \to H^2_E(\hX; \Omega^{n-1}_{\hX}(\log E)).$$ Since $1$-Du Bois lci singularities are rational,   $H^1_E(\hX; \Omega^{n-1}_{\hX}(\log E)) =0$  by \cite[1.8]{FL} and  the $1$-Du Bois assumption implies that $H^2_E(\hX; \Omega^{n-1}_{\hX}(\log E))=0$ (cf.\ \cite[2.8]{FL}). Hence $H^0(X;T^1_X) \cong H^1(\hX; \Omega^{n-1}_{\hX}(\log E))$. The global case is similar, using $\mathbb{T}^1_Y \cong H^1(V;\Omega^{n-1}_{\hY}(\log E)|V)$, and the compatibility is clear.
\end{proof}

There is a similar result for $1$-rational singularities:

\begin{lemma}\label{0.2} If the singularities of $X$ are isolated $1$-rational lci singularities, then $H^0(X;T^1_X) \cong H^1(\hX; \Omega^{n-1}_{\hX}(\log E)(-E))$. Globally,
$\mathbb{T}^1_Y \cong H^1(\hY; \Omega^{n-1}_{\hY}(\log E)(-E))$, and there is a commutative diagram
$$\begin{CD}
  H^1(\hY; \Omega^{n-1}_{\hY}(\log E)(-E)) @>>> H^0(Y;R^1\pi_*\Omega^{n-1}_{\hY}(\log E)(-E)) \\
  @V{\cong}VV @VV{\cong}V \\
\mathbb{T}^1_Y @>>>  H^0(Y;T^1_Y) .
\end{CD}$$
\end{lemma}
\begin{proof} Since isolated $1$-rational singularities are $1$-Du Bois, it suffices by Lemma~\ref{0.2.1} to show that the map $H^1(\hX; \Omega^{n-1}_{\hX}(\log E)(-E)) \to H^1(\hX; \Omega^{n-1}_{\hX}(\log E))$ is an isomorphism. We have the long exact sequence
\begin{gather*}
 H^0(E; \Omega^{n-1}_{\hX}(\log E)|E) \to H^1(\hX; \Omega^{n-1}_{\hX}(\log E)(-E)) \to H^1(\hX; \Omega^{n-1}_{\hX}(\log E)) \\
 \to H^1(E; \Omega^{n-1}_{\hX}(\log E)|E).
 \end{gather*}
Moreover, $H^1(E; \Omega^{n-1}_{\hX}(\log E)|E)= \Gr_F^{n-1}H^n(L)$, which has dimension $\ell^{n-1, 1}  = \ell^{1, n-2} =0$ by the $1$-rational condition \cite[Theorem 5.3(iv)]{FL22d}. Likewise $\dim H^0(E; \Omega^{n-1}_{\hX}(\log E)|E) =  \ell^{n-1, 0}$.   Since $X$ is a rational singularity, $\ell^{n-1, 0}  = 0$  by a result of Steenbrink \cite[Lemma 2]{SteenbrinkDB}. Hence 
$$H^1(\hX; \Omega^{n-1}_{\hX}(\log E)(-E)) \cong H^1(\hX; \Omega^{n-1}_{\hX}(\log E)).$$ The global case and the compatibility are again clear.
\end{proof}

\begin{remark}\label{remark1.3}  In the global case, where we do not make the assumption that $\omega_Y \cong \scrO_Y$, the above lemmas remain true provided that we replace $H^1(\hY; \Omega^{n-1}_{\hY}(\log E))$ resp.\  $ H^1(\hY; \Omega^{n-1}_{\hY}(\log E)(-E))$ by $H^1(\hY; \Omega^{n-1}_{\hY}(\log E)\otimes\pi^*\omega_Y^{-1})$ resp.\  $ H^1(\hY; \Omega^{n-1}_{\hY}(\log E)(-E)\otimes \pi^*\omega_Y^{-1})$.
\end{remark}

To illustrate how these results may be used in practice, we give a quick proof of a slight variant of  \cite[Corollary 5.8]{FL}:

 \begin{theorem}\label{1limthm} Suppose that $Y$ is a canonical Calabi--Yau variety of dimension $n\ge 3$ with isolated $1$-liminal hypersurface singularities. Then a strong first-order smoothing of $Y$ exists $\iff$ for every $x\in Z$, there  exists $a_x\in \Cee$, $a_x\neq 0$, such that $\sum_xa_x\varphi(\varepsilon_x)=0$ in $H^2(\hY; \Omega^{n-1}_{\hY})$, where  $\varepsilon _x \in \Gr^{n-1}_FH^n (L_x)$ is a generator and $\varphi$ is the composition
  $$H^1(E; \Omega^{n-1}_{\hY}(\log E)|E) \xrightarrow{\partial}  H^2(\hY; \Omega^{n-1}_{\hY}(\log E)(-E)) \to H^2(\hY; \Omega^{n-1}_{\hY}).$$
 In particular, if $Y$ satisfies the above condition,  it is smoothable. 
  \end{theorem}
  \begin{proof}  By Lemma~\ref{0.2.1}, there are isomorphisms
  $$H^0(X;T^1_X)\cong H^1(U; T^0_X|U)\cong H^1(U;\Omega^{n-1}_{\hX}(\log E)|U).$$ 
  Following the isomorphism  $H^0(X;T^1_X)\cong H^1(U;\Omega^{n-1}_{\hX}(\log E)|U)$ with the restriction map 
  $$H^1(U;\Omega^{n-1}_{\hX}(\log E)|U ) \to H^1(E; \Omega^{n-1}_{\hY}(\log E)|E)$$ gives a homomorphism 
  $H^0(Y; T^1_Y) \to H^1(E; \Omega^{n-1}_{\hY}(\log E)|E)$, such that the following  diagram is commutative: 
  $$\begin{CD}
  \mathbb{T}^1_Y @>>> H^0(Y; T^1_Y) @. \\
  @V{\cong}VV @VVV @.\\
   H^1(\hY; \Omega^{n-1}_{\hY}(\log E)) @>>> H^1(E; \Omega^{n-1}_{\hY}(\log E)|E) @>{\partial}>>  H^2(\hY; \Omega^{n-1}_{\hY}(\log E)(-E)).
  \end{CD}$$
  Here, if as usual $E_x = \pi^{-1}(x)$,  $H^1(E_x; \Omega^{n-1}_{\hY}(\log E)|E_x)$ has dimension one for every $x\in Z$  by the $1$-liminal assumption.  Let $\varepsilon_x$ be a basis vector. By \cite[Lemma 2.6, Theorem 2.1(v)]{FL}, the map $T^1_{Y,x}   \to H^1(E_x; \Omega^{n-1}_{\hY}(\log E)|E_x)$ is surjective and its kernel is $\mathfrak{m}_xT^1_{Y,x}$. Thus, $Y$ has a strong first-order smoothing $\iff$ for every $x\in Z$, there exists $a_x\in \Cee$, $a_x\neq 0$, such that $\sum_{x\in Z}a_x \partial(\varepsilon_x)=0$ in $H^2(\hY; \Omega^{n-1}_{\hY}(\log E)(-E))$. 
 By Lemma~\ref{lemma3.3}, the map  $H^2(\hY; \Omega^{n-1}_{\hY}(\log E)(-E)) \to H^2(\hY; \Omega^{n-1}_{\hY})$ is injective. It follows that  $\sum_xa_x\partial(\varepsilon_x)=0$ in $H^2(\hY; \Omega^{n-1}_{\hY}(\log E)(-E))$ $\iff$ $\sum_xa_x\varphi(\varepsilon_x)=0$ in $H^2(\hY; \Omega^{n-1}_{\hY})$.  Thus a strong first-order smoothing exists $\iff$ $\sum_xa_x\varphi(\varepsilon_x)=0$. The final statement then follows from \cite[Corollary 1.5]{FL22c}.
 \end{proof}
 
 \begin{remark}\label{1limFano} There is a similar result in the $1$-liminal Fano case:   Assume that  $Y$ has only  isolated $1$-liminal hypersurface singularities  and that $\omega_Y^{-1}$ is ample.  In this case, the above construction produces an obstruction to a strong first-order smoothing, namely  $\sum_{x\in Z} c_x\lambda_x^k \partial(\varepsilon_x) \in H^2(\hY; \Omega^{n-1}_{\hY}(\log E)(-E)\otimes \pi^*\omega_Y^{-1})$. The group $H^2(\hY; \Omega^{n-1}_{\hY}(\log E)(-E)\otimes \pi^*\omega_Y^{-1})$ is Serre dual to $H^{n-2}(\hY; \Omega^1_{\hY}(\log E) \otimes \pi^*\omega_Y)$.   In many  cases,  $H^{n-2}(\hY; \Omega^1_{\hY}(\log E) )\otimes \pi^*\omega_Y)=0$. For example, if there exists a smooth Cartier divisor $H$ on $Y$, thus not passing through the singular points of $Y$, such that $\omega_Y =\scrO_Y(-H)$, and in addition $H^{n-3}(H; \Omega^1_H) = 0$, then an argument with the Goresky-MacPherson-Lefschetz theorem in intersection cohomology \cite{GoreskyMacPherson} shows that 
  $$H^{n-2}(\hY; \Omega^1_{\hY}(\log E)\otimes \pi^*\omega_Y)=H^{n-2}(\hY; \Omega^1_{\hY}(\log E)\otimes \scrO_{\hY}(-H))= 0,$$
  where we identify the divisor $H$ on $Y$ with its preimage $\pi^*H$ on $\hY$.
The proof of Theorem~\ref{1limthm} then shows  that, under these assumptions,  a strong first-order smoothing of $Y$ always exists, and hence $Y$ is smoothable by \cite[Theorem  4.5]{FL}. A somewhat stronger statement is proved in \cite[Corollary 4.10]{FL}. 
 \end{remark}
 
 \begin{remark} In dimension three, a singular point can be $k$-liminal only for $k=1$. Since this case is covered by Theorem~\ref{1limthm}, we are free to  make the assumption that $n \ge 4$ as needed in what follows. 
  \end{remark}

\subsection{Weighted homogeneous singularities and quotient stacks}\label{ss1.3} For the remainder of this section, we are concerned with generalizing the above picture, and in particular Lemma~\ref{0.2},  in the  context of stacks: Assume that the isolated singularity $X$ is  locally  analytically isomorphic to a weighted cone in $\Cee^{n+1}$  over a weighted hypersurface $E\subseteq WP^n$. Thus we may as well assume that $X$ is the weighted cone as in Lemma~\ref{ex0.4}(iii), with an isolated singularity at $0$. 

\begin{definition}\label{def1.8}   Let $X$ be the weighted cone  in $\Cee^{n+1}$  over a weighted hypersurface $E\subseteq WP^n$, where $WP^n$ is a weighted projective space, and  $X^{\#}$  the weighted blowup of $X$ as in \cite[\S3]{FL}.  Let  $\uE$, $\uhX$, $\uP^n$ be the corresponding  quotient stacks. If $X$ has an isolated singularity at $0$, then  $\uhX$ and $\uE$ are  quotient stacks for an action of $\Cee^*$  on smooth schemes with finite stabilizers.  Hence $\uhX$ is a smooth stack, $\uE$ is a smooth divisor in $\uhX$, and there is a morphism $\uhX \to X$ which defines an isomorphism $\uhX - \uE \to X -\{0\}$. 

Globally, let $Y$ be a projective variety  of dimension $n$ with isolated   weighted homogeneous hypersurface singularities. Let   $\pi\colon  Y^{\#} \to Y$   denote the weighted blowup of $Y$ at the singularities,  and let $E$ be the exceptional divisor, i.e.\ $E = \pi^{-1}(Z)$ where $Z = Y_{\text{\rm{sing}}}$. We can also construct a stacky version of $Y^{\#}$ as follows: For each $x\in Z$, we have the corresponding exceptional divisor $E_x$.  Let $X$ denote the corresponding weighted cone in $\Cee^{n+1}$.
There is a (Zariski) open neighborhood $U\subseteq Y$ of $x$ and an \'etale morphism  $U\to X$.  We can then pull back the stack $\uhX$  to a stack $ \underline{U}^{\#}$ and glue $\underline{U}^{\#}$ and $Y-\{x\}$ along the Zariski open subset $U-\{x\}$. Doing this for each singular point defines the stack $\uhY$. 

A similar construction works in the analytic category, where we view an analytic stack as a functor on the category of complex analytic spaces.  This allows for the possibility that, in  Definition~\ref{defnCY}, $Y$ is a compact analytic, not necessarily algebraic space. 
\end{definition}

As in Definition~\ref{def1.8}, let $X^{\#}$ be the weighted blowup of $X$, with $\uhX$ the associated stack, and let $\hX$ be an arbitrary  log resolution. Given a projective $Y$ with isolated  weighted homogeneous hypersurface singularities, we define $\uhY$ as before and let $\pi\colon \hY \to Y$ be a log resolution. To avoid confusion, we denote the exceptional divisor of  $\pi\colon \hX \to X$ or $\pi\colon \hY \to Y$  by $\widehat{E}$. We claim that, in the statement of  Lemmas~\ref{0.2.1} and \ref{0.2}, we can replace ordinary cohomology with stack cohomology.  First, we recall the following definition, due to Steenbrink \cite[\S1]{Steenvanishing}, \cite[\S2]{SteenCompositio}:

\begin{definition}\label{Steensheaves}  Let $W$ be an analytic space which is an orbifold ``viewed as an analytic space,"  i.e.\ locally $W =  \widetilde{W}/G$, where $G$ is a small subgroup  of $GL(n,\Cee)$ in the sense of  \cite{SteenCompositio} and $\widetilde{W}$ is a $G$-invariant neighborhood of the origin on $\Cee^n$. Let  $W_0$ be the open subset where $W$ is (locally) a free quotient so  that, by hypothesis,  $W-W_0$ had codimension at least $2$. Define $\Omega^p_{W}$ to be $i_*\Omega^p_{W_0}$, where $i\colon W_0 \to W$ is the inclusion. If $\pi\colon \widehat{W} \to W$ is a resolution of singularities, then  $\Omega^p_{W} =\pi_*\Omega^p_{\widehat{W}}$.  If (locally) $W =  \widetilde{W}/G$ as above, then $\Omega^p_{W} = (\Omega^p_{\widetilde{W}})^G$.  If $D$ is an orbifold normal crossing divisor of $W$ in the obvious sense, then $\Omega^p_{W}(\log D)$ is defined similarly. 

By \cite[(1.9), (1.12)]{Steenvanishing}, the complex $(\Omega^\bullet_{W},d)$ is a resolution of the constant sheaf $\Cee$ and, if $W$ is projective, the hypercohomology spectral sequence with $E_1^{p,q} = H^q(W; \Omega^p_{W}) \implies \mathbb{H}^{p+q}(W;\Omega^\bullet_{W}) \cong H^{p+q}(W; \Cee)$ degenerates at $E_1$.  Likewise, if $D$ is an orbifold normal crossing divisor of $W$, then   $\mathbb{H}^k(W;\Omega^\bullet_{W}(\log D))\cong H^k(W-D; \Cee)$ and the analogous spectral sequence also degenerates at $E_1$. 
\end{definition}

There is  an extension of Lemma~\ref{biratinv} to this situation:

\begin{lemma}\label{stackisoms} In the  notation of Definition~\ref{def1.8}, for all $p,q$, there are isomorphisms
\begin{gather*}
H^q(\uhX; \Omega^p_{\uhX}(\log \uE)) \cong H^q(X^{\#}; \Omega^p_{X^{\#}}(\log E)) \cong H^q(\hX; \Omega^p_{\hX}(\log \widehat{E}));\\
H^q(\uhY; \Omega^p_{\uhY}(\log \uE)) \cong H^q(Y^{\#}; \Omega^p_{Y^{\#}}(\log E)) \cong H^q(\hY; \Omega^p_{\hY}(\log \widehat{E})),
\end{gather*}
where $\Omega^p_{X^{\#}}(\log E)$ and $\Omega^p_{Y^{\#}}(\log E)$ are the sheaves defined in Definition~\ref{Steensheaves}  for the spaces $X^{\#}$ and $Y^{\#}$. Likewise, with similar definitions of $\Omega^p_{X^{\#}}(\log E)(-E)$ and $\Omega^p_{Y^{\#}}(\log E)(-E)$, 
\begin{gather*}
H^q(\uhX; \Omega^p_{\uhX}(\log \uE)(-E)) \cong H^q(X^{\#}; \Omega^p_{X^{\#}}(\log E)(-E)) \cong H^q(\hX; \Omega^p_{\hX}(\log \widehat{E})(-\widehat{E}));\\
H^q(\uhY; \Omega^p_{\uhY}(\log \uE)(-E)) \cong H^q(Y^{\#}; \Omega^p_{Y^{\#}}(\log E)(-E)) \cong H^q(\hY; \Omega^p_{\hY}(\log \widehat{E})(-\widehat{E})).
\end{gather*}
\end{lemma}
\begin{proof} These statements follow from  the arguments of \cite[Lemma 3.13, Lemma 3.14]{FL} and Lemma~\ref{biratinv}.
\end{proof} 

Thus for example in the situation of Lemma~\ref{0.2}, we have the following:

\begin{corollary}\label{0.2.2} If all of the singularities of $Y$ are weighted homogeneous isolated $1$-rational singularities, then there is a commutative diagram 
$$\begin{CD}
H^1(\uhY; \Omega^{n-1}_{\uhY}(\log \uE)(-\uE)) @>>> H^0(Y;R^1\pi_*\Omega^{n-1}_{\uhY}(\log \uE)(-\uE))\\
  @V{\cong}VV @VV{\cong}V \\
  H^1(\hY; \Omega^{n-1}_{\hY}(\log \widehat{E})(-\widehat{E})) @>>> H^0(Y;R^1\pi_*\Omega^{n-1}_{\hY}(\log \widehat{E})(-\widehat{E})) \\
  @V{\cong}VV @VV{\cong}V \\
\mathbb{T}^1_Y @>>>  H^0(Y;T^1_Y) .
\end{CD}$$
Here $H^0(Y;R^1\pi_*\Omega^{n-1}_{\uhY}(\log \uE)(-\uE))$ is a direct sum of terms isomorphic to the corresponding local terms $H^1(\uhX; \Omega^{n-1}_{\uhX}(\log \uE)(-\uE))$. \qed
\end{corollary}

In the local setting, we note  the following for future reference:

\begin{lemma}\label{PRseq} There is an exact sequence
$$0 \to \Omega^k_{\uE} \to \Omega^k_{\uhX}(\log \uE)|\uE \to \Omega^{k-1}_{\uE} \to 0.$$
\end{lemma} 
\begin{proof} Poincar\'e residue induces a surjection $\Omega^1_{\hX}(\log \uE)|\uE \to \scrO_{\uE}$ whose kernel is easily checked to be $\Omega^1_{\uE}$  as $\uE$ is smooth. Taking the $k^{\text{\rm{th}}}$ exterior power gives the exact sequence. 
\end{proof}

\begin{remark}\label{linknotation} With $E$ as in Definition~\ref{def1.8}, we can either think of $E$ as a scheme or as a stack. We will denote by  $H^i(E) = H^i(E;\Cee)$   the usual singular cohomology. By the remarks at the end of Definition~\ref{Steensheaves}, there is a spectral sequence $E_1^{p,q} = H^q(E; \Omega^p_E) \implies H^{p+q}(E; \Cee)$ and it degenerates at $E_1$. Moreover, the corresponding filtration  defines a (pure) Hodge structure on $H^i(E)$ \cite{Steenvanishing}. The method of proof of \cite[Lemma 3.13]{FL} shows that $H^q(E; \Omega^p_E) \cong H^q(\uE; \Omega^p_{\uE})$.  Thus in particular
$$\Gr_F^pH^{p+q}(E) \cong H^q(\uE; \Omega^p_{\uE}).$$
As noted in the introduction, the cohomology of the link $L$ of the isolated singularity $X$ carries a mixed Hodge structure. (We will not try to give a stacky interpretation of $L$.) Arguments as in the case where $E$ is smooth show that
$$\Gr_F^p H^{p+q}(L) \cong H^q(\uE; \Omega^p_{\uhX}(\log \uE)|\uE).$$
\end{remark}

\section{Local calculations}\label{s2}

\subsection{Numerology}   In this section, we consider the local case.   We keep the notation of the previous section: $X$ is the affine weighted cone over a hypersurface $E$ in a weighted projective space $WP^n$, with an isolated singularity at $0$, and $X^{\#}$ is the weighted blowup, with $\uhX$, $\uE$, and $\uP^n$ the corresponding stacks.      Let $a_1, \dots, a_{n+1}$ be the $\Cee^*$ weights, $d$ the degree of $E$, and set $w_i = a_i/d$. Setting $N = \sum_ia_i -d$,  as a line bundle on the stack $\uE$, 
  $$K_{\uE}  =\scrO_{\uE}(-N) = \scrO_{\uE}(d -\sum_ia_i).$$
Since $\sum_iw_i = N/d + 1$, the $k$-liminal condition is equivalent to:
  $$k = \sum_iw_i -1 = N/d \iff N = dk.$$
  Thus $K_{\uE}  =\scrO_{\uE}(-dk)$. As for $K_{\uhX}$, we have $K_{\uhX} = \scrO_{\uhX}(r\uE)$ for some $r\in \Zee$. By adjunction,
  $$K_{\uE}  =\scrO_{\uE}(-dk) = K_{\uhX} \otimes \scrO_{\uhX}(\uE)|\uE = \scrO_{\uhX}((r+1)\uE)|\uE = \scrO_{\uE}(-(r+1)\uE).$$
  Thus $r+1 = dk$, $r = dk -1$, and 
  $$K_{\uhX} = \scrO_{\uhX}((dk -1)\uE)  = \scrO_{\uhX}((N-1)\uE).$$
  
  To simplify the  notation,  set
  $$a = d(k-1) = N-d = \sum_ia_i -2d.$$
  Thus $a=0$ $\iff$ $k=1$, i.e.\ $X$ is $1$-liminal. Moreover, 
  $$K_{\uE}(a) = K_{\uE}  \otimes \scrO_{\uE}(a)=\scrO_{\uE}(-d).$$
  
  \subsection{Some cohomology calculations} 
  
\begin{assumption}\label{ass2.1} From now on, we assume that $X$ is a $k$-liminal weighted homogeneous isolated hypersurface singularity with $n=\dim X \ge 4$ and $k \ge 2$.  In particular, $X$ is $1$-rational, so that Lemma~\ref{0.2} and Corollary~\ref{0.2.2} apply.  
\end{assumption}

 \begin{lemma}\label{0.4} With notation as above, if $j\le 2$ and $1\le i \le a-1$, then 
 $$H^j(\uE; \Omega^{n-1}_{\uE}(i))  =  H^j(\uE; \Omega^{n-2}_{\uE}(i))  = 0.$$
 For $i=a$, we have  $H^j(\uE; \Omega^{n-1}_{\uE}(a))=0$ for $j\le 2$ and $H^j(\uE; \Omega^{n-2}_{\uE}(a)) =0$ for $j=0,2$, but $\dim H^1(\uE; \Omega^{n-2}_{\uE}(a))  = \dim H^1(\uE; T_{\uE}(-d)) = 1$.
 \end{lemma}
 \begin{proof}  First, 
 $H^j(\uE; \Omega^{n-1}_{\uE}(i)) =  H^j(\uE; K_{\uE}(i)) = H^j(\uE; \scrO_{\uE}(-kd+i))$. 
 We have the exact sequence
 $$0\to \scrO_{\uP^n}(r) \to \scrO_{\uP^n}(r+d) \to \scrO_{\uE}(r+d) \to 0.$$
 Since $H^i(\uP^n; \scrO_{\uP^n}(r))=0$ for $i=1,2,3$ and all $r$,  $H^j(\uE; K_{\uE}(i)) =0$ for $j=1,2$ and all $i$. For $j=0$,  since $0\le i \le a-1 = kd-d-1$,  $-kd+i\le -d-1 < 0$, and hence $H^0(\uE; K_{\uE}(i)) = H^0(\uE; \scrO_{\uE}(-kd+i)) =0$ in this range as well.
 
 For $H^j(\uE; \Omega^{n-2}_{\uE}(i))$, note first that, as $E$ has dimension $n-1$, 
 $$\Omega^{n-2}_{\uE}(i) \cong T_{\uE} \otimes K_{\uE}(i) = T_{\uE}(-kd+i).$$
 From the  normal bundle sequence
 $$0 \to T_{\uE} \to T_{\uP^n}|\uE \to \scrO_{\uE}(d) \to 0,$$
 we therefore obtain
 $$0\to T_{\uE}(-kd+i) \to T_{\uP^n}(-kd+i)|\uE \to \scrO_{\uE}(-kd+d+i)\to 0.$$
 For $i \le a-1$, $-kd+d+i \le -1$. Then an argument as before shows that, for $j\le 2$, 
 $$H^j(\uE; \Omega^{n-2}_{\uE}(i)) \cong H^j(\uE; T_{\uP^n}(-kd+i)|\uE).$$
 We have  the Euler exact sequence
 $$0 \to \scrO_{\uE} \to \bigoplus_{i=1}^{n+1}\scrO_{\uE}(a_i)  \to T_{\uP^n}|\uE \to 0.$$
 Still assuming that $j\le 2$ and $i\le a-1$, it suffices to show that
 $$H^{j+1}(\uE; \scrO_{\uE}(-kd+i)) = H^j(\uE; \scrO_{\uE}(-kd+i+a_i))=0$$
 for $j\le 2$.  This is certainly true if $n\ge 5$, again using   $-kd+i+a_i \le a_i-d  \le -1$  since $X$ is not smooth and hence $a_i < d$. For $n=4$, 
  $H^3(\uE; \scrO_{\uE}(-kd+i))= H^3(\uE;K_{\uE}(i))$
  which is Serre dual to $H^0(\uE; \scrO_{\uE}(-i))$ so we are done as before since $i\ge 1$.
  
  To prove the second statement, note that $\Omega^{n-1}_{\uE}(a) = K_{\uE}(a) = \scrO_{\uE}(-d)$ and   $H^j(\uE; \scrO_{\uE}(-d)) =0$ for $j \le 2$ by the same reasons as before. Likewise, $\Omega^{n-2}_{\uE}(a) \cong T_{\uE} \otimes K_{\uE}(a) = T_{\uE}(-d)$. Via the Euler exact sequence
   $$0 \to \scrO_{\uE}(-d) \to \bigoplus_{i=1}^{n+1}\scrO_{\uE}(a_i-d)  \to T_{\uP^n}(-d)|\uE \to 0,$$
   we see that $H^j(\uE; T_{\uP^n}(-d)|\uE)=0$ for $j\le 2$. Moreover the normal bundle sequence gives
   $$0 \to T_{\uE}(-d) \to T_{\uP^n}(-d)|\uE \to \scrO_{\uE} \to 0.$$
   Thus  $H^0(\uE; T_{\uE}(-d)) = H^2(\uE; T_{\uE}(-d))=0$ but the coboundary map $H^0(\scrO_{\uE}) \to H^1(\uE; T_{\uE}(-d))$ is an isomorphism. 
 \end{proof}
 
  \begin{corollary}\label{0.4.1} Under Assumption~\ref{ass2.1}, 
  \begin{enumerate}
  \item[\rm(i)]   $H^0(\uE; \Omega^{n-1}_{\uhX}(\log \uE)(-i\uE)|\uE)=0$ for $1 \le i\le a$;
   \item[\rm(ii)] $H^1(\uE; \Omega^{n-1}_{\uhX}(\log \uE)(-i\uE)|\uE)=0$ for $1\le i< a$;
    \item[\rm(iii)]  $\dim H^1(\uE;\Omega^{n-1}_{\uhX}(\log \uE) (-a\uE)|\uE) = 1$. 
    \end{enumerate} 
 \end{corollary}
 \begin{proof} 
   By Lemma~\ref{PRseq},  there is an exact sequence
   $$0 \to \Omega^{n-1}_{\uE}(i)  \to \Omega^{n-1}_{\uhX}(\log \uE)(-i\uE)|\uE  \to \Omega^{n-2}_{\uE}(i) \to 0.$$
  By   Lemma~\ref{0.4}, if $1\le i\le a$, then $H^0(\uhX; \Omega^{n-1}_{\uhX}(\log \uE)(-i\uE)|\uE)=0$,   and  
  $$H^1(\uE; \Omega^{n-1}_{\uhX}(\log \uE)(-i\uE)|\uE )   \cong  H^1(\uE;\Omega^{n-2}_{\uE}(i) ),$$ which is $0$ for $1\le i< a$ and has dimension $1$ for $i=a$. 
 \end{proof} 

 \begin{theorem}\label{0.7.1} Under Assumption~\ref{ass2.1},  
 $$  H^0(X;T^1_X) \cong  H^1(\uhX; \Omega^{n-1}_{\uhX}(\log \uE) ( -a\uE) ).$$
  Moreover, the natural map 
 $H^1(\uhX; \Omega^{n-1}_{\uhX}(\log \uE) ( -a\uE))  \to H^1(\uE;\Omega^{n-1}_{\uhX}(\log \uE)  ( -a\uE) |\uE)$ induces an isomorphism 
 $$H^0(X;T^1_X)/\mathfrak{m}_xH^0(X;T^1_X) \cong H^1(\uE;\Omega^{n-1}_{\uhX}(\log \uE) ( -a\uE) |\uE) .$$
 \end{theorem}
 \begin{proof} For the first part, we have an exact sequence
 $$0 \to \Omega^{n-1}_{\uhX}(\log \uE) ( -(i+1)\uE) \to \Omega^{n-1}_{\uhX}(\log \uE) ( -i\uE) \to \Omega^{n-1}_{\uhX}(\log \uE) ( -i\uE) |\uE   \to 0.$$
 Thus, by Corollary~\ref{0.4.1}, for $1\le i< a$ we have an isomorphism 
 $$H^1(\uhX; \Omega^{n-1}_{\uhX}(\log \uE) ( -(i+1)\uE) \to H^1(\uhX; \Omega^{n-1}_{\uhX}(\log \uE) ( -i\uE))$$
 and by induction, starting with the isomorphism $H^0(X;T^1_X)\cong H^1(\uhX; \Omega^{n-1}_{\uhX}(\log \uE)(-\uE))$ of Lemma~\ref{0.2} and Corollary~\ref{0.2.2}, we see that $H^0(X;T^1_X) \cong  H^1(\uhX; \Omega^{n-1}_{\uhX}(\log \uE) ( -a\uE) )$.  
 
 To see the final statement, we have an exact sequence
 $$H^1(\uhX; \Omega^{n-1}_{\uhX}(\log \uE)( -(a+1)\uE)) \to H^1(\uhX; \Omega^{n-1}_{\uhX}(\log \uE)( -a\uE))\to H^1(\uhX;\Omega^{n-1}_{\uhX}(\log \uE)  ( -a\uE) |\uE),$$
 and hence an injection 
 $$H^1(\uhX; \Omega^{n-1}_{\uhX}(\log \uE) ( -a\uE))\Big/\im H^1(\uhX; \Omega^{n-1}_{\uhX}(\log \uE) ( -(a+1)\uE))\to H^1(\uhX;\Omega^{n-1}_{\uhX}(\log \uE)  ( -a\uE) |\uE).$$
 By Corollary~\ref{0.4.1}(iii),  $\dim H^1(\uE;\Omega^{n-1}_{\uhX}(\log \uE) ( -a\uE) |\uE) = 1$. Thus, if the map 
 $$H^0(X;T^1_X)/\mathfrak{m}_xH^0(X;T^1_X) \to H^1(\uE;\Omega^{n-1}_{\uhX}(\log \uE)  ( -a\uE) |\uE)$$ is nonzero, it is an isomorphism. However, to prove that this map is nonzero, it is necessary to consider the $\Cee^*$ picture as in \cite[\S3]{FL}: The vector bundle $\Omega^{n-1}_{\uhX}(\log \uE) $ on  $\uhX$ is of the form $\rho^*W$ for some vector bundle $W$ on $\uE$, where $\rho\colon \uhX \to \uE$ is the natural morphism, and $\scrO_{\uhX}(-\uE) = \rho^*\scrO_{\uE}(1)$.  Then
 $$ H^1(\uhX;\Omega^{n-1}_{\uhX}(\log \uE)) = \bigoplus_{r\ge 0}H^1(\uE; W(r)) = \bigoplus_{r\ge -N}H^0(X;T^1_X)(r) = \bigoplus_{r\ge -d}H^0(X;T^1_X)(r).$$
Here, the final equality holds because $-d$ is the smallest weight occurring in $H^0(X;T^1_X)$ and $-N = -dk \le -d$.  Note also that $\bigoplus_{r\ge -d+1}H^0(X;T^1_X)(r) = \mathfrak{m}_xH^0(X;T^1_X)$. 
 Taking the tensor product with $\scrO_{\uhX}(-i\uE)$ has the effect of shifting the weight spaces by $i$ since 
 $$\rho^*W \otimes \scrO_{\uhX}(-i\uE) \cong \rho^*W \otimes \rho^*\scrO_{\uE}(i) = \rho^*(W\otimes \scrO_{\uE}(i)).$$
  Thus 
 $$ H^1(\uhX;\Omega^{n-1}_{\uhX}(\log \uE)( -i\uE)) = \bigoplus_{r\ge 0}H^1(\uE; W(r+i)) = \bigoplus_{r\ge i}H^1(\uE; W(r))= \bigoplus_{r \ge -N+i}H^0(X;T^1_X)(r).$$
Here $-d \ge -N + i$ $\iff$ $i \le N -d = a$. This recovers the fact  that $H^1(\uhX; \Omega^{n-1}_{\hX}(\log \uE)( -i\uE)) \cong H^0(X;T^1_X)$ for $i \le a$, whereas 
$$H^1(\uhX; \Omega^{n-1}_{\uhX}(\log \uE) ( -(a+1)\uE)) = \bigoplus_{r \ge -d+1}H^0(X;T^1_X)(r) = \mathfrak{m}_xH^0(X;T^1_X)$$
as claimed. 
  \end{proof} 
  
\subsection{Definition of the nonlinear map}  We now consider the analogue of \cite[ Lemma 4.10]{RollenskeThomas}. First, we have the subsheaf $T_{\uhX}(-\log \uE) \subseteq T_{\uhX}$ (on the stack $\uhX$) which is the kernel of the map $T_{\uhX}\to N_{\uE/\uhX}$ or equivalently is dual to $\Omega^1_{\uhX}(\log \uE)$.    There is thus a commutative diagram
 $$\begin{CD}
 T_{\uhX}(-\log \uE) @>>> T_{\uhX}\\
 @V{\cong}VV @VV{\cong}V \\
  \Omega^{n-1}_{\uhX}(\log \uE)(-\uE)  \otimes K_{\uhX}^{-1} @>>>  \Omega^{n-1}_{\uhX}  \otimes K_{\uhX}^{-1}
 \end{CD}$$
  There are compatible isomorphisms 
 \begin{align*}
 T_{\uhX}(d\uE) &\cong \Omega^{n-1}_{\uhX}\otimes K_{\uhX}^{-1} \otimes \scrO_{\uhX}(  d\uE)  =  \Omega^{n-1}_{\uhX}((d-dk +1)\uE)=    \Omega^{n-1}_{\uhX}(-a\uE + \uE);\\
 T_{\uhX}(-\log \uE)(d\uE)  &\cong \Omega^{n-1}_{\uhX}(\log\uE)(-\uE) \otimes K_{\uhX}^{-1} \otimes \scrO_{\uhX}(  d\uE)  \\
 &=   \Omega^{n-1}_{\uhX}(\log \uE)((d-dk )\uE)=    \Omega^{n-1}_{\uhX}(\log \uE)(-a\uE).
 \end{align*}
 Taking $k^{\text{\rm{th}}}$ exterior powers,    $\bigwedge ^kT_{\uhX}$ is dual to $\Omega^k_{\uhX}$ and hence is isomorphic to $\Omega^{n-k}_{\uhX}\otimes K_{\uhX}^{-1} $ and  $\bigwedge ^kT_{\uhX}(-\log \uE)$ is dual to $\Omega^k_{\uhX}(\log \uE)$ and hence is isomorphic to $\Omega^{n-k}_{\uhX}(\log \uE)(-\uE)\otimes K_{\uhX}^{-1} $.  There are compatible isomorphisms
  \begin{gather*} 
   \bigwedge ^k(T_{\uhX}(d\uE)) \cong \Omega^{n-k}_{\uhX}\otimes K_{\uhX}^{-1}\otimes\scrO_{\uhX}(dk\uE)= \Omega^{n-k}_{\uhX}\otimes\scrO_{\uhX}(\uE);   \\
  \bigwedge ^k(T_{\uhX}(-\log \uE)(d\uE)) \cong \Omega^{n-k}_{\uhX}(\log \uE)(-\uE)\otimes K_{\uhX}^{-1}\otimes\scrO_{\uhX}(dk\uE)= \Omega^{n-k}_{\uhX} (\log \uE).
    \end{gather*}

 So we have a commutative diagram
 $$\begin{CD}
 \bigwedge ^k(T_{\uhX}(-\log \uE)(d\uE)) \cong  \bigwedge ^k\left(  \Omega^{n-1}_{\uhX}(\log \uE)(-a\uE)\right) @>>>  \bigwedge ^k\left(  \Omega^{n-1}_{\uhX}(\log \uE)(-a\uE)|\uE\right)\\
 @V{\cong}VV   @VV{\cong}V \\
 \Omega^{n-k}_{\uhX} (\log \uE) @>>> \Omega^{n-k}_{\uhX} (\log \uE)|\uE.
 \end{CD}$$
  There is also the induced map $ T_{\uhX}(-\log \uE) \to T_{\uE}$, and the following commutes:
  $$\begin{CD}
  T_{\uhX}(-\log \uE) @>>> T_{\uE} \\ 
  @V{\cong}VV @VV{\cong}V\\
 \Omega^{n-1}_{\uhX}(\log \uE)(-\uE) \otimes  K_{\uhX}^{-1} @>{\operatorname{Res}}>> \Omega^{n-2}_{\uE} \otimes  K_{\uE}^{-1},
\end{CD}$$
using the adjunction isomorphism $\scrO_{\uhX}(-\uE)\otimes  K_{\uhX}^{-1}|\uE = ( K_{\uhX} \otimes \scrO_{\uhX}(\uE))^{-1}|\uE \cong K_{\uE}^{-1}$.

  The  exact sequence of Lemma~\ref{PRseq} yields an exact sequence
$$0 \to \Omega^{n-1}_{\uE}(a) \to \Omega^{n-1}_{\uhX}(\log \uE)(-a\uE)|\uE \to \Omega^{n-2}_{\uE} (a)\to 0.$$
By Lemma~\ref{0.4}, there is an induced  isomorphism $H^1(\uE; \Omega^{n-1}_{\uhX}(\log \uE)(-a\uE)|\uE) \to H^1(\uE; \Omega^{n-2}_{\uE} (a))$. Moreover, 
$$\Omega^{n-2}_{\uE} (a) \cong T_{\uE} \otimes K_{\uE}(a) = T_{\uE}(-kd+a) = T_{\uE}(-d).$$
Taking $k^{\text{th}}$ exterior powers,
$$\bigwedge^k(T_{\uE}(-d)) = \left(\bigwedge^kT_{\uE}\right)(-kd) = \left(\bigwedge^kT_{\uE}\right)\otimes K_{\uE} \cong \Omega^{n-k-1}_{\uE}.$$
A combination of wedge product and cup product induces  symmetric homogeneous  degree $k$ maps 
\begin{gather*}
\nu_{\uhX}\colon H^1(\uhX; \Omega^{n-1}_{\uhX}(\log \uE)(-a\uE)) \to H^k(\uhX;\Omega^{n-k}_{\uhX} (\log \uE));\\
\mu_{\uhX}= \nu_{\uhX} |\uE\colon H^1(\uE; \Omega^{n-1}_{\uhX}(\log \uE)(-a\uE)|\uE) \to H^k(\uE;\Omega^{n-k}_{\uhX} (\log \uE)|\uE),
\end{gather*}
 and a commutative diagram  (with nonlinear vertical maps)
 
$$\begin{CD}
 H^1(\uhX; \Omega^{n-1}_{\uhX}(\log \uE)(-a\uE)) @>>> H^1(\uE; \Omega^{n-1}_{\uhX}(\log \uE)(-a\uE)|\uE) \\
@V{\nu_{\uhX}}VV @VV{\mu_{\uhX}}V\\
 H^k(\uhX;\Omega^{n-k}_{\uhX} (\log \uE)) @>>> H^k(\uE;\Omega^{n-k}_{\uhX} (\log \uE)|\uE) .
\end{CD}$$
There are similarly  compatible  symmetric homogeneous  degree $k$ maps 
\begin{gather*}
\nu_{\uhX}' \colon H^1(\uhX; T_{\uhX}(-\log \uE)(d\uE) ) \to H^k(\uhX; \bigwedge ^k(T_{\uhX}(-\log \uE)(d\uE)) )\cong H^k(\uhX;\Omega^{n-k}_{\uhX} (\log \uE));\\
\mu_{\uhX}'\colon H^1(\uE; T_{\uE}(-d)) \cong  H^1(\uE; \Omega^{n-2}_{\uE} (a)) \to H^k(\uE; \bigwedge^k(T_{\uE}(-d))) \cong H^k(\uE; \Omega^{n-k-1}_{\uE}).
\end{gather*}
The following  diagram with nonlinear vertical maps commutes:
$$\begin{CD}
 H^1(\uE; \Omega^{n-1}_{\uhX}(\log \uE)(-a\uE)|\uE)  @>{\cong}>> H^1(\uE; T_{\uE}(-d)) \cong  H^1(\uE; \Omega^{n-2}_{\uE} (a))\\
 @V{\mu_{\uhX}}VV @VV{\mu_{\uhX}'}V\\
H^k(\uE;\Omega^{n-k}_{\uhX} (\log \uE)|\uE) @>{\operatorname{Res}}>> H^k(\uE; \bigwedge^k(T_{\uE}(-d))) \cong H^k(\uE; \Omega^{n-k-1}_{\uE}).
 \end{CD}$$

By Lemma~\ref{0.4}, Corollary~\ref{0.4.1}(iii) and Theorem~\ref{introlink}, 
\begin{align*}
\dim H^1(\uE; \Omega^{n-1}_{\uhX}(\log \uE)(-a\uE)|\uE) &= \dim  H^1(\uE; T_{\uE}(-d))=1;\\
\dim H^k(\uE;\Omega^{n-k}_{\uhX} (\log \uE)|\uE) &= \dim \Gr^{n-k}_FH^n(L) = 1.
\end{align*}
The map $H^k(\uhX;\Omega^{n-k}_{\uhX} (\log \uE)) \to H^k(\uE; \Omega^{n-k-1}_{\uE})$ factors through the (surjective) map 
$$H^k(\uhX;\Omega^{n-k}_{\uhX} (\log \uE)) \to H^k(\uE;\Omega^{n-k}_{\uE} (\log \uE)|\uE) = \Gr^{n-k}_FH^n(L),$$   and the map 
\begin{align*} H^k(\uE;\Omega^{n-k}_{\uhX} (\log \uE)|\uE) &= \Gr^{n-k}_FH^n(L) \\
\to H^k(\uE;\Omega^{n-k-1}_{\uE}) &= \Gr^{n-k-1}_FH^{n-1}(E) = \Gr^{n-k}_FH^{n-1}(E)(-1)\end{align*}
is an isomorphism in almost all cases. More precisely, let $H_0^{n-1}(E)$ be the primitive cohomology of $E$ in dimension $n-1$ and let  $ H_0^{n-1-k}(\uE;\Omega^k_{\uE}) = \Gr^k_FH_0^{n-1}(E)$ be the corresponding groups. 

\begin{lemma}\label{lemma2.4}  If $X$ is not an ordinary double point or if $n$ is even, then the map $\Gr^{n-k}_FH^n(L )\to H^k(\uE;\Omega^{n-k-1}_{\uE})$ is an isomorphism, and hence 
$$\dim H^k(\uE;\Omega^{n-k-1}_{\uE}) = 1.$$ 
If   $X$ is an ordinary double point and $n=2k+1$ is odd, then 
$$\Gr^{n-k}_FH^n(L )\to H^k(\uE;\Omega^{n-k-1}_{\uE})=\Gr^{n-k-1}_FH^{n-1}(E)$$ is injective with image     $ H_0^k(\uE;\Omega^k_{\uE}) =\Cee([A]-[B])$, and hence $\dim H_0^k(\uE;\Omega^k_{\uE}) = 1$ . 
\end{lemma}
\begin{proof} As noted in Definition~\ref{def0.3}, if $L$ is the link of the singularity, then $\dim \Gr^{n-k}_FH^n(L) =1$.  Then Lemma~\ref{PRseq} gives the exact sequence
$$\Gr^{n-k}_FH^n(E) \to \Gr^{n-k}_FH^n(L) \to \Gr^{n-k-1}_FH^{n-1}(E) \to \Gr^{n-k}_FH^{n+1}(E).$$
Since $E$ is an orbifold  weighted hypersurface in $WP^n$, $\Gr^i_FH^j(E)=0$ except for the cases $j=2i$ or $i+j = n-1$. Thus  $\Gr^{n-k}_FH^n(E) = \Gr^{n-k-1}_FH^{n-1}(E) =0$ unless $n = 2(n-k)$, i.e.\ $k=\frac12n$, or $n+1 = 2(n-k)$, i.e.\ $k=\frac12(n-1)$. The first case is excluded since we assumed that $X$ is a singular point and the second case only arises if $n=2k+1$ and $X$ is an ordinary double point (Lemma~\ref{ex0.4}). This proves the first statement, and the second statement is the well-known computation of the primitive cohomology of an even-dimensional quadric. 
\end{proof}

 \begin{proposition}\label{0.9} The map $\mu_{\uhX}$ is not $0$. Hence there exist bases  $v\in H^1(\uE; \Omega^{n-1}_{\uhX}(\log \uE)(-a\uE)|\uE)$ and $\varepsilon \in H^k(\uE;\Omega^{n-k}_{\uhX} (\log \uE)|\uE)$ of the two one-dimensional vector spaces and a nonzero $c\in \Cee$ such that, for all $\lambda \in \Cee$,
 $$\mu_{\uhX}(\lambda v) = c\lambda^k\varepsilon.$$
 \end{proposition}
 \begin{proof} It suffices to prove that the map $\mu_{\uhX}'$ is nonzero. Taking the $(i+1)^{\text{\rm{st}}}$ exterior power  of the normal bundle sequence 
 $$0 \to T_{\uE}(-d) \to (T_{\uP^n}|\uE)(-d) \to \scrO_{\uE}  \to 0$$ gives   exact sequences
 \begin{equation*}
 0 \to  \bigwedge^{i+1}T_{\uE} (-(i+1)d)\to  \bigwedge^{i+1}(T_{\uP^n}|\uE) (-(i+1)d) \to  \bigwedge^iT_{\uE} (-id)\to 0,\tag{$*$}
 \end{equation*}
 and thus a sequence of connecting homomorphisms
 $$\partial_i \colon H^i(\uE;  \bigwedge^iT_{\uE} (-id)) \to H^{i+1}(\uE;  \bigwedge^{i+1}T_{\uE} (-(i+1)d).$$
 We claim the following: 
 
 \begin{claim}\label{0.10}  There exists a nonzero element $\eta \in H^1(\uE;T_{\uE}(-d))$, necessarily a generator,  such that $\mu'_{\uhX}(\eta)  = \pm\partial_{k-1} \circ \cdots \circ   \partial_1(\eta)$.
 \end{claim}
 
 \begin{claim}\label{0.11} The connecting homomorphism $\partial_i$ is an isomorphism for $1\leq i \leq k-2$ and injective for $i=k-1$.
 \end{claim}
 Clearly the two claims imply Proposition~\ref{0.9}.
 \renewcommand{\qedsymbol}{}
 \end{proof}
 
 \begin{proof}[Proof of Claim~\ref{0.10}] The element $\eta =\partial_0(1)\in H^1(\uE;T_{\uE}(-d))$ is the extension class for the extension $0 \to T_{\uE}(-d) \to (T_{\uP^n} |\uE)(-d) \to \scrO_{\uE} \to 0$. By the last line of the proof of Lemma~\ref{0.4}, the coboundary map  $\partial_0$ is injective and hence $\eta \neq 0$. Then a calculation shows that, up to sign,  
 $$\wedge \eta \in H^1(\uE; Hom( \bigwedge^iT_{\uE} (-id),  \bigwedge^{i+1}T_{\uE} (-(i+1)d)))$$
 is the corresponding extension class for the extension $(*)$. Since the connecting homomorphism is given by cup product with the extension class, we see that 
 $$\mu'_{\uhX}(\eta) = \eta^k =  \pm\partial_{k-1} \circ \cdots \circ   \partial_1(\eta) \in H^k(\uE;  \bigwedge^kT_{\uE} (-kd)).\qed$$
 \renewcommand{\qedsymbol}{}
 \end{proof}
 
 \begin{proof}[Proof of Claim~\ref{0.11}] It suffices to show that $H^i(\uE;  \bigwedge^{i+1}(T_{\uP^n}|\uE) (-(i+1)d)) =0$ for $1\le i \leq k-1$. First note that
 $$ \bigwedge^{i+1}(T_{\uP^n}|\uE) (-(i+1)d)) = \Big(\Omega^{n-i-1}_{\uP^n}|\uE\Big)  \Big(\sum_ka_k -(i+1)d\Big).$$
 We have the exact sequence
 $$0 \to \Omega^\ell_{\uP^n}(r-d) \to \Omega^\ell_{\uP^n}(r)\to  (\Omega^\ell_{\uP^n}|\uE) (r) \to 0.$$
 By Bott vanishing (or directly), $H^j(\uE;  (\Omega^\ell_{\uP^n}|\uE) (r)) = 0$ as long as $1\le j\leq n-2$ and $j\neq \ell$ or $\ell+1$. In our situation, $i\leq k-1 < k \le \frac12(n-1)$, and thus $i < n-i-1$. In particular, $i \neq n-i-1$ or $n-i$. Thus $H^i(\uE;  \bigwedge^{i+1}(T_{\uP^n}|\uE) (-(i+1)d)) =0$. 
 \end{proof}
 
 \begin{remark} The above calculations are connected with the computation of the Hodge filtration on $E$. For example, in case $X$ is a cone over the smooth degree $d$ hypersurface $E$ in $\Pee^n$, then  $H^0(\Pee^n; K_{\Pee^n}\otimes (n-k)d)  =H^0(\Pee^n;\scrO_{\Pee^n}(-n-1 + d(k+1)) = H^0(\Pee^n;\scrO_{\Pee^n})$ is identified via residues with $F^{n-k}H^n(L) = \Gr_F^{n-k}H^n(L)$.
 \end{remark}

 \section{The global setting}\label{s3}
 
\subsection{Deformation theory}\label{ss3.1} We assume the following for the rest of this subsection:

\begin{assumption}\label{ass3.1} $Y$ is a canonical Calabi--Yau variety of dimension $n\ge 4$, all of whose singularities are $k$-liminal isolated weighted homogeneous hypersurface singularities, with $k \ge 2$, as the case $k=1$ has already been considered in Theorem~\ref{1limthm}.  We freely use  the notation of the previous sections, especially that of Definition~\ref{def1.8}. In particular, $Y^{\#}$ is the weighted blowup at each point $x$ of $Z=Y_{\text{\rm{sing}}}$, with exceptional divisor $E_x$, and $a_x$ is the integer defined in \S2.1.  We let $\uhY$ and $\uE = \sum_{x\in Z} \uE_x$ be the associated stacks. Let $\vec{a}\uE$ denote the divisor $\sum_{x\in Z}a_x\uE_x$. 
\end{assumption}

  The argument of Theorem~\ref{0.7.1} shows:
 
\begin{lemma} There is a commutative diagram
 $$\begin{CD}
 \mathbb{T}^1_Y  @>{\cong}>>  H^1(\uhY; \Omega^{n-1}_{\uhY}(\log \uE) ( -\vec{a}\uE)) \\
 @VVV @VVV\\
 H^0(Y; T^1_Y)@>{\cong}>> H^0(Y; R^1\pi_*\Omega^{n-1}_{\uhY}(\log \uE) ( -\vec{a}\uE)).\qed
 \end{CD}$$ 
\end{lemma}

We also have the subsheaf $T_{\uhY}(-\log \uE) \subseteq T_{\uhY}$. As in \S2, globally there is an isomorphism
 $$ \bigwedge ^k\left(  \Omega^{n-1}_{\uhY}(\log \uE)(-\vec{a}\uE)\right) \cong \Omega^{n-k}_{\uhY} (\log \uE).$$

 Then the global form  of the discussion in \S\ref{s2}  yields:

 \begin{theorem}\label{mainCD} There is  a commutative diagram 
 $$\begin{CD}
H^1(\uhY; \Omega^{n-1}_{\uhY}(\log \uE) ( -\vec{a}\uE)) @>>> H^1(\uE;\Omega^{n-1}_{\uhY}(\log \uE)( -\vec{a}\uE)|\uE )\\
 @V{\nu_{\uhY}}VV @VV{\mu_{\uhY}}V  \\
 H^k(\uhY; \Omega^{n-k}_{\uhY}(\log \uE)) @>>>  H^k(\uE; \Omega^{n-k}_{\uhY}(\log\uE)|\uE) .\qed
 \end{CD}$$
 \end{theorem}

 Here $\mu_{\uhY}$ is the sum of the local maps $\mu_{\uhX}$ at each component of $E$ and $\nu_{\uhY}$ is also a  homogeneous map of degree $k$. Note that,   after we localize at a singular point $x$ of $Y$,   
 $$\dim H^k(\uE_x; \Omega^{n-k}_{\uhX}(\log \uE_x)|\uE_x) = \dim \Gr^{n-k}_FH^n(L_x) = 1.$$
 By Corollary~\ref{0.4.1}(iii), $\dim H^1(\uE_x;\Omega^{n-1}_{\uhY}(\log \uE ) ( -a_x\uE)|\uE_x ) =1$ as well, and so the $\mu_{\uhY}$ in the diagram, at each singular point $x$ of $Y$, is a homogeneous degree $k$ map between two one-dimensional vector spaces.  For every $x\in Z$, fix  an isomorphism  $H^1(\uE_x;\Omega^{n-1}_{\uhY}(\log \uE ) ( -a_x\uE)|\uE_x ) \cong \Cee$, i.e.\ a basis vector $v_x \in  H^1(\uE_x;\Omega^{n-1}_{\uhY}(\log \uE ) ( -a_x\uE)|\uE_x ) $, and a basis vector $\varepsilon_x  \in \Gr^{n-k}_FH^n(L_x)$.  It follows by  Proposition~\ref{0.9} that, for every $x\in Z$,  there exists a  nonzero $c_x\in \Cee$, depending only on the above choices,  such that, for every $\lambda = (\lambda_x) \in \Cee^Z\cong H^1(\uE;\Omega^{n-1}_{\uhY}(\log \uE)( -\vec{a}\uE)|\uE )$,
 $$\mu_{\uhY}(\lambda) = \sum_{x\in Z} c_x\lambda_x^k \varepsilon_x.$$ 
 
 Consider the following diagram, where the vertical arrows are homogeneous of degree $k$ and the bottom row is exact:
 $$\begin{CD}
H^1(\Omega^{n-1}_{\uhY}(\log \uE)( -\vec{a}\uE)) @>>> H^1(\Omega^{n-1}_{\uhY}(\log \uE)( -\vec{a}\uE)|\uE )@. \\
 @V{\nu_{\uhY}}VV @VV{\mu_{\uhY}}V @. \\
 H^k( \Omega^{n-k}_{\uhY}(\log \uE)) @>>>  H^k(\Omega^{n-k}_{\uhY}(\log \uE)|\uE) @>{\partial}>>  H^{k+1}(\Omega^{n-k}_{\uhY}(\log \uE)(-\uE)).
 \end{CD}$$
The above diagram then implies the following: if a class 
$$\alpha = (\alpha_x) \in H^1(\uE; \Omega^{n-1}_{\uhY}(\log \uE)( -\vec{a}\uE)|\uE )$$ is   the image of $\beta \in H^1(\uhY; \Omega^{n-1}_{\uhY}(\log \uE)( -\vec{a}\uE))$, then $\mu_{\uhY}(\alpha)$ is   the image of $$\nu_{\uhY}(\beta) \in H^k(\uhY;\Omega^{n-k}_{\uhY}(\log \uE)), $$ and  hence $\partial(\mu_{\uhY}(\alpha)) = 0$ in $H^{k+1}(\Omega^{n-k}_{\uhY}(\log \uE)(-\uE))$.

Returning to the world of spaces, as opposed to stacks, consider  a log resolution $\pi \colon \hY \to Y$ with exceptional divisor which we continue to denote by $E$. Then via the isomorphism 
$$H^{k+1}(\uhY; \Omega^{n-k}_{\uhY}(\log \uE)(-\uE)) \cong H^{k+1}(\hY; \Omega^{n-k}_{\hY}(\log E)(-E))$$ of Lemma~\ref{stackisoms},    the coboundary $\partial(\mu_{\uhY}(\alpha))$ defines  an element of $H^{k+1}(\hY; \Omega^{n-k}_{\hY}(\log E)(-E))$. Moreover, if $\partial(\mu_{\uhX}(\alpha))$ is of the form  $\sum_{x\in Z} c_x\lambda_x^k \partial(\varepsilon_x)$, then it has the same form when viewed as an element of $H^{k+1}(\hY; \Omega^{n-k}_{\hY}(\log E)(-E))$, by the commutativity of the diagram 
$$\begin{CD}
H^k(\uE; \Omega^{n-k}_{\uhY}(\log \uE)|\uE) @>{\partial}>>  H^{k+1}(\uhY; \Omega^{n-k}_{\uhY}(\log \uE)(-\uE))\\
@V{\cong}VV  @VV{\cong}V\\
H^k(E; \Omega^{n-k}_{\hY}(\log E)|E) @>{\partial}>>  H^{k+1}(\hY; \Omega^{n-k}_{\hY}(\log E)(-E)).
\end{CD}$$

So finally we obtain:

\begin{theorem}\label{mainthma} For every  $x\in Z =Y_{\text{\rm{sing}}}$, fix    isomorphisms 
$$H^0(T^1_{Y,x})/\mathfrak{m}_xH^0(T^1_{Y,x}) \cong  H^1(E_x;\Omega^{n-1}_{\hX}(\log E_x)( -a_xE_x)|E _x) \cong \Cee.$$
Write $\Gr^{n-k}_FH^n(L_x) = H^k(E_x; \Omega^{n-k}_{\hY}(\log E_x)|E_x) =\Cee\cdot \varepsilon_x$ for some fixed choice of a generator $\varepsilon_x$.  For all $x\in Z$, there exist   $c_x\in \Cee^*$ depending only on the above identifications, with the following property: Suppose that the class $(\bar\theta_x) \in \bigoplus_{x\in Y_{\text{\rm{sing}}}}H^0(T^1_{Y,x})/\mathfrak{m}_xH^0(T^1_{Y,x}) $ is in the image of $\theta \in \mathbb{T}^1_Y$, and let $\lambda_x \in \Cee $ be the complex number  corresponding  to $\bar\theta_x$  via the above identification. 
If $\partial \colon H^k(E; \Omega^{n-k}_{\hY}(\log E)|E)\to H^{k+1}(\hY; \Omega^{n-k}_{\hY}(\log E)(-E))$ is the coboundary map, then
 $$\sum_{x\in Z} c_x\lambda_x^k \partial(\varepsilon_x) =0. \qed$$
 \end{theorem}   

 We can post-compose  the coboundary map 
 $$\partial \colon \Gr^{n-k}_FH^n(L) = H^k(E; \Omega^{n-k}_{\hY}(\log E)|E) \to H^{k+1}(\hY; \Omega^{n-k}_{\hY}(\log E)(-E))$$
 with the natural (injective) map 
 $$H^{k+1}(\hY; \Omega^{n-k}_{\hY}(\log E)(-E)) \to H^{k+1}(\hY; \Omega^{n-k}_{\hY}).$$
 Let $\varphi\colon \Gr^{n-k}_FH^n(L) = H^k(E; \Omega^{n-k}_{\hY}(\log E)|E) \to H^{k+1}(\hY; \Omega^{n-k}_{\hY})$ be the above composition. This is the same as the induced map on $\Gr^{n-k}_F$ of the natural map $H^n(L) \to H^{n+1}(\hY)$ which is the Poincar\'e dual of the map $H_{n-1}(L) \to H_{n-1}(\hY)$. 

\begin{corollary}\label{last}  With  the notation and hypotheses of Theorem~\ref{mainthma},  and with $\varphi\colon \Gr^{n-k}_FH^n(L) \to H^{k+1}(\hY; \Omega^{n-k}_{\hY})$ the natural map as above,  the following holds in $H^{k+1}(\hY; \Omega^{n-k}_{\hY})$:
$$\sum_{x\in Z} c_x\lambda_x^k\varphi(\varepsilon_x) =0.$$ 
In particular, if a strong first-order smoothing of $Y$ exists, then for all $x\in Z$ there exists $\lambda_x\in \Cee^*$  with $\sum_{x\in Z} c_x\lambda_x^k\varphi(\varepsilon_x) =0$. \qed
\end{corollary}

\begin{remark} (i) By Poincar\'e duality, the map $H^n(L) \to H^{n+1}(\hY)$ is the same as the map $H_{n-1}(L) \to H_{n-1}(\hY)$, which factors as $H_{n-1}(L) \to H_{n-1}(Y) \to H_{n-1}(\hY)$.  By Remark~\ref{duality}, we can identify $\partial \colon \Gr^{n-k}_FH^n(L)  \to H^{k+1}(\hY; \Omega^{n-k}_{\hY}(\log E)(-E))$ with the corresponding map
$$\Gr^{n-k}_FH_{n-1}(L)(-n) \to \Gr^{n-k}_FH_{n-1}(Y)(-n).$$
This gives an equivalent statement to Theorem~\ref{mainthma} which only involves $Y$, not the choice of a resolution. 

\smallskip
\noindent (ii) By Lemma~\ref{lemma3.3}, the map $H^{k+1}(\hY; \Omega^{n-k}_{\hY}(\log E)(-E)) \to H^{k+1}(\hY; \Omega^{n-k}_{\hY})$ is injective.
 Thus $\sum_{x\in Z} c_x\lambda_x^k\varphi(\varepsilon_x) =0$ $\iff$ $\sum_{x\in Z} c_x\lambda_x^k \partial(\varepsilon_x) =0$, so that Theorem~\ref{mainthma} and Corollary~\ref{last} contain the same information.
 \end{remark}
 
 \begin{remark} It is certainly possible for $\partial(\varepsilon_x) =0$. For example, suppose that  $\dim Y= 2k+1$ and the singularities of $Y$ are all $k$-liminal, i.e.\ ordinary double points. If $Y$ is a Calabi--Yau hypersurface in $\Pee^{2k+2}$ with just  a few singular points in general position, then they can be smoothed independently, i.e.\ for every $x\in Z$, there exists a $\theta \in \mathbb{T}^1_Y$ such that $\lambda_x\neq 0$ but $\lambda_{x'}= 0$ for all $x'\neq x$.  Then $\partial(\varepsilon_x) =0$ for every $x\in Z$.
  \end{remark}

  \begin{remark} As in Remark~\ref{1limFano}, we can also consider the Fano case, where $Y$ has isolated $k$-liminal weighted homogeneous hypersurface singularities and $\omega_Y^{-1}$ is ample.  In this case, the construction produces an obstruction to a strong first-order smoothing, namely  $\sum_{x\in Z} c_x\lambda_x^k \partial(\varepsilon_x) \in H^{k+1}(\hY; \Omega^{n-k}_{\hY}(\log E)(-E)\otimes \pi^*\omega_Y^{-1})$. By Serre duality, $H^{k+1}(\hY; \Omega^{n-k}_{\hY}(\log E)(-E)\otimes \pi^*\omega_Y^{-1})$ is dual to $H^{n-k-1}(\hY; \Omega^k_{\hY}(\log E) \otimes \pi^*\omega_Y)$. 
  
  In many reasonable cases, however, $H^{n-k-1}(\hY; \Omega^k_{\hY}(\log E)\otimes \pi^*\omega_Y)=0$. For example, if there exists a smooth Cartier divisor $H$ on $Y$, thus not passing through the singular points, such that $\omega_Y =\scrO_Y(-H)$, and in addition $H^{n-k-2}(H; \Omega^k_H) = 0$, then the argument of  Remark~\ref{1limFano} shows that 
  $$H^{n-k-1}(\hY; \Omega^k_{\hY}(\log E)\otimes \pi^*\omega_Y)=H^{n-k-1}(\hY; \Omega^k_{\hY}(\log E)\otimes \scrO_{\hY}(-H))= 0,$$
  where we identify the divisor $H$ on $Y$ with its preimage $\pi^*H$ on $\hY$.  For example, these hypotheses are satisfied if $Y$ is a hypersurface in $\Pee^{n+1}$ of degree $d\leq n+1$. However, as soon as $n=2k+1$ is odd and $n\ge 5$, there exist such hypersurfaces with only nodes as singularities (the $k$-liminal case with $k = \frac12(n-1)$) such that the map $\mathbb{T}^1_Y \to H^0(Y; T^1_Y)= \bigoplus_{x\in Y_{\text{\rm{sing}}}}H^0(T^1_{Y,x})/\mathfrak{m}_xH^0(T^1_{Y,x}) $ is not surjective (cf.\ for example \cite[Remark 4.11(iv)]{FL}). Thus, the obstructions to the surjectivity of the map $\mathbb{T}^1_Y \to H^0(Y; T^1_Y)$ are not detected by the nonlinear obstruction $\sum_{x\in Z} c_x\lambda_x^k \partial(\varepsilon_x)$.  Of course, a nodal hypersurface in $\Pee^{n+1}$ is smoothable, but the above examples show that, even in  the  Fano case,  the nodes cannot always be smoothed independently.
  \end{remark}
  
  \subsection{Geometry of a smoothing}\label{ss3.2}  We make the following assumption throughout this subsection (except for Remark~\ref{remark313} at the end):
  
  \begin{assumption}  $Y$ denotes a \emph{projective} variety, not necessarily satisfying $\omega^{-1}_Y$ ample or $\omega_Y \cong \scrO_Y$,  with only isolated lci singular points (not necessarily weighted homogeneous). Denote by $Z$ the singular locus of $Y$.  Let $f\colon \mathcal{Y} \to \Delta$  be a \emph{projective} smoothing of $Y$, i.e.\ $Y_0 = Y \cong f^{-1}(0)$ and the remaining fibers $Y_t = f^{-1}(t)$, $t\neq 0$, are smooth. For $x\in Z$,  let $L_x$ denote the link at $x $ and let $M_x$ denote the Milnor fiber  at $x$. Finally, let  $M =\bigcup_{x\in Z}M_x$ and $L= \bigcup_{x\in Z}L_x$.
  \end{assumption}

We have the Mayer--Vietoris sequence of mixed Hodge structures (where $H^i(Y_t)$ is given the limiting mixed Hodge structure): 
\begin{equation}\label{MVMilnor}
 \cdots \to H^{i-1}(M) \to  H^i(Y,Z) \to H^i(Y_t) \to H^i(M) \to \cdots .
 \end{equation}
In particular, just under the assumption that  $Y$ has isolated lci singularities, $H^i(Y,Z) \to H^i(Y_t)$ is an isomorphism except for the cases $i =n, n+1$.  There is a more precise result if we assume that  the singularities are $k$-Du Bois:

\begin{lemma}\label{first}  Suppose that all singular points of $Y$ are isolated lci $k$-Du Bois singularities. Then    
\begin{enumerate} 
\item[\rm(i)] $\Gr^p_FH^n(M_x)=0$ for $p\le k$ and   $\Gr^{n-p}_FH^n(M_x)=0$ for $p\le k-1$.
\item[\rm(ii)]  For all $i$, if  $p\le k$, then $\Gr^p_FH^i(Y_t) = \Gr^p_FH^i(Y)$ and 
 if  $p\le k-1$, then $\Gr^{n-p}_FH^i(Y_t) = \Gr^{n-p}_FH^i(Y)$.
\end{enumerate}   
\end{lemma} 
\begin{proof} The first statement follows from  \cite[\S6]{FL22d} and the second  from (i), (\ref{MVMilnor}),  and strictness. 
\end{proof} 

\begin{remark} Under the assumption of  isolated lci $k$-Du Bois singularities as above (or more generally  isolated lci $(k-1)$-rational singularities), the above implies that $\Gr^W_{2n-a}H^n(Y_t) = 0$   for all $a\le 2k-1$, and hence that, for all $a\le 2k-1$, $\Gr^W_aH^n(Y_t) = 0$  as well. Thus, if $T$ is the monodromy operator acting on $H^n(Y_t)$ and $N = \log T^m$ for a sufficiently divisible power of $T$, then $N^{n-2k+1} =0$. 
\end{remark}

 Under the assumption of  $k$-liminal singularities, the proof of  Lemma~\ref{first} and \cite[Corollary 6.14]{FL22d}  give the following:

\begin{lemma}\label{second} In the above notation, if all singular points of $Y$ are isolated $k$-liminal hypersurface singularities, then 
$$\Gr^{n-k}_FH^n(M_x) \cong \Gr^{n-k}_FH^n(L_x) =\Cee \cdot \varepsilon_x$$
for some nonzero $\varepsilon_x \in \Gr^{n-k}_FH^n(L_x)$.
 Moreover,  there is an exact sequence
$$0 \to \Gr^{n-k}_FH^n(Y) \to \Gr^{n-k}_FH^n(Y_t)\to \bigoplus_{x\in Z}\Cee \cdot \varepsilon_x \xrightarrow{\psi}   \Gr^{n-k}_FH^{n+1}(Y) \to \Gr^{n-k}_FH^{n+1}(Y_t)\to 0.\qed$$
\end{lemma}

We also have the natural map $\varphi\colon \Gr^{n-k}_FH^n(L) =\bigoplus_{x\in Z}\Cee \cdot \varepsilon_x\to \Gr^{n-k}_FH^{n+1}(\hY)$, and 
there is a commutative diagram
$$\begin{CD}
\Gr^{n-k}_FH^n(M) @>{\cong}>> \Gr^{n-k}_FH^n(L) \\
@V{\psi}VV @VV{\varphi}V \\
\Gr^{n-k}_FH^{n+1}(Y) @>>> \Gr^{n-k}_FH^{n+1}(\hY).
\end{CD}$$
By   Lemma~\ref{lemma3.3},   $\Gr^{n-k}_FH^{n+1}(Y) \to \Gr^{n-k}_FH^{n+1}(\hY)$ is injective. 
Thus the dimension of the kernel and image of the map $\psi \colon  \bigoplus_{x\in Z}\Cee \cdot \varepsilon_x \to  \Gr^{n-k}_FH^{n+1}(Y)$ are the same as the dimensions of the kernel and image of the  map $\varphi \colon  \bigoplus_{x\in Z}\Cee \cdot \varepsilon_x \to  \Gr^{n-k}_FH^{n+1}(\hY)$. Then we have the following generalization of  \cite[Lemma 8.1(2)]{FriedmanSurvey}:

\begin{corollary}\label{third} Still assuming that  all singular points of $Y$ are isolated $k$-liminal hypersurface singularities, in the above notation, let $s' = \dim \Ker \{\varphi  \colon \bigoplus_{x\in Z}\Cee \cdot \varepsilon_x \to  \Gr^{n-k}_FH^{n+1}(\hY)\}$ and let $s'' =\#(Z) - s'= \dim \im \varphi$. Then:
\begin{enumerate}
\item[\rm(i)]  $h^{n-k, k}(Y_t) = h^{k, n-k}(Y_t)= \dim \Gr^{n-k}_FH^n(Y_t) = \dim \Gr^{n-k}_FH^n(Y) +s'$.
\item[\rm(ii)] $\dim \Gr^k_FH^n(Y)  = \dim \Gr^{n-k}_FH^n(Y) +s'$.
\item[\rm(iii)] $h^{n-k,k+1}(Y_t)= \dim \Gr^{n-k}_FH^{n+1}(Y_t) = \dim \Gr^{n-k}_FH^{n+1}(Y) -s''$.  
\end{enumerate}
\end{corollary}
\begin{proof}  (i) and (iii) follow from the exact sequence in Lemma~\ref{second}. As for (ii), by Lemma~\ref{first}(ii), 
$$\dim \Gr^k_FH^n(Y)  =  \dim \Gr^k_FH^n(Y_t)  =\dim \Gr^{n-k}_FH^n(Y_t) = \dim \Gr^{n-k}_FH^n(Y) +s',$$
using (i). 
\end{proof}

\begin{remark}\label{remark313} 
Let  $Y$ be a compact analytic  threefold with    all singular points   $1$-liminal, hence  ordinary double points.   Assume in addition   that $h^1(\scrO_Y) = h^2(\scrO_Y)=0$. For a canonical Calabi--Yau threefold, since $\omega_Y\cong \scrO_Y$, this is a natural assumption to make: If  $h^1(\scrO_Y) \neq 0$, $Y$ is smooth by a result of Kawamata \cite[Theorem 8.3]{KawamataFiber}, and  $h^1(\scrO_Y) = 0$ $\iff$ $h^2(\scrO_Y) =h^2(\omega_Y) =0$ by Serre duality. Let $Y'$ be   a small resolution of $Y$, and let $[C_x]\in H^2(Y'; \Omega^2_{Y'})=H^4(Y')$ be the fundamental class of the exceptional curve over the point $x\in Z$. Setting $\psi \colon \Cee^Z \to H^2(Y'; \Omega^2_{Y'})$ to be the natural map $(a_x) \mapsto \sum_{x\in Z}a_x[C_x]$, let $s'=\dim \Ker \psi$ and $s''=\dim \im \psi$. Then  arguments similar to those above show that 
$$ b_4(Y) = b_4(Y') = b_2(Y') = b_2(Y) + s''.$$  Moreover, if $Y$ is smoothable and $Y_t$ denotes a general smoothing , then:
\begin{align*}
 b_2(Y_t) &= b_2(Y) = b_2(Y') - s'';\\ 
 b_3(Y_t) &= b_3(Y') + 2s'. 
\end{align*}
In particular, if $Y$ is a $1$-liminal canonical   Calabi--Yau threefold and   $\psi=0$, or equivalently $s'' =0$ in the above notation, i.e.\ $Y$ is $\Q$-factorial,  then $Y$ is smoothable by Theorem~\ref{dim3criterion} and the Kawamata--Ran--Tian theorem, and the above formulas hold for $Y_t$. 
\end{remark}

\bibliography{cyref}
\end{document}